\documentclass[11pt]{amsart}

\usepackage[T1]{fontenc}
\usepackage{amsmath,amssymb,amsthm}
\usepackage{tikz}
\usepackage{tikz-cd}
\usepackage{enumerate}  
\usepackage{mathrsfs}
\usepackage{graphicx}
\usepackage{bm}
\usepackage{verbatim}
\usepackage{hyperref}
\usepackage[usenames,dvipsnames]{xcolor}
\newtheorem{theorem}{Theorem}[section]
\newtheorem*{theorem*}{Theorem}
\newtheorem{proposition}[theorem]{Proposition}
\newtheorem{lemma}[theorem]{Lemma}
\newtheorem{corollary}[theorem]{Corollary}
\theoremstyle{definition}
\newtheorem{definition}[theorem]{Definition}
\theoremstyle{definition}
\newtheorem{defthm}[theorem]{Definition/Theorem}
\theoremstyle{remark}
\newtheorem{remark}[theorem]{Remark}
\theoremstyle{remark}
\newtheorem{example}[theorem]{Example}
\theoremstyle{definition}
\newtheorem{notation}[theorem]{Notation}
\theoremstyle{conjecture}
\newtheorem{conjecture}[theorem]{Conjecture}
\newtheorem*{conjecture*}{Conjecture}
\newtheorem{assumption}[theorem]{Assumption}

\usepackage[normalem]{ulem}

\newcommand{\NS}{\mathrm{NS}}
\newcommand{\MW}{\mathrm{MW}}
\newcommand{\Pic}{\mathrm{Pic}}

\newcommand{\Q}{\mathbb Q}

\def\PP{{\mathbb P}}

\title[Bounds on Mordell-Weil ranks]{Bounds on the Mordell-Weil rank of elliptic  fibrations}

\author[Grassi]{Antonella Grassi\textsuperscript{1}}
\address{\textsuperscript{1} Dipartimento di Matematica and INFN, Universit\`a di Bologna,  Bologna, Italy}
\address{\textsuperscript{1} Department of Mathematics, University of Pennsylvania, Philadelphia, USA}
 \email{antonella.grassi3@unibo.it}
 
\author[Miranda]{Rick Miranda\textsuperscript{2}}
\address{\textsuperscript{2} Department of Mathematics, Colorado State University, Fort Collins (CO), 80523, USA}
\email{rick.miranda@colostate.edu}

\author[Paranjape]{Kapil Paranjape\textsuperscript{3}}
\address{\textsuperscript{3} Department of Mathematics, Indian Institute of Science Education and
Research, Mohali, Sector 81, SAS Nagar, Mohali, Punjab 140306, India}
\email{kapil@iisermohali.ac.in}

\author[Srinivas]{Vasudevan Srinivas\textsuperscript{4}}
\address{\textsuperscript{4}Department of Mathematics, University at Buffalo - SUNY, Buffalo, NY 14260, USA}
\email{vs74@buffalo.edu}

\author[Weigand]{Timo Weigand\textsuperscript{5}}
\address{\textsuperscript{5}II. Institut f\"ur Theoretische Physik, Universit\"at Hamburg, Notkestrasse 9, 22607 Hamburg, Germany} 
\address{\textsuperscript{5}Zentrum f\"ur Mathematische Physik, Universit\"at Hamburg, Bundesstrasse 55, 20146 Hamburg, Germany }
\email{timo.weigand@desy.de}

\date{\today}

\begin{document}

\begin{flushright}
 \small{ZMP-HH/26-6} \vspace{4mm}
\end{flushright}

\begin{abstract} We prove that the Mordell–Weil group of a higher dimensional elliptic fibration naturally embeds into that of a suitable elliptic surface. We give sufficient conditions for the existence of such  surfaces. \color{black} We apply our result to obtain explicit and  uniform bounds for the Mordell–Weil rank of elliptic threefolds of Kodaira dimension zero, including Calabi–Yau threefolds, confirming predictions from physics. We prove new explicit bounds for a broad class of elliptic fourfolds. These results suggest a general linear bound for the Mordell–Weil rank in terms of the dimension of the elliptic fibration, which we formulate as a conjecture.

\end{abstract}

\maketitle

\section{Introduction}

  It is known that the rank of the Mordell-Weil group  of a  smooth  elliptic
  K3 surface can only take values 
  $0 \leq {\rm rk}\MW(X/B)\leq 18$,  and that each such value can be obtained \cite{Cox1982, Kuwata2000MW,Kloosterman2007}. These are  the elliptic surfaces with section, singular fibers and $\kappa(X)=0$.

Unlike the case of elliptic surfaces, very little is known  about bounds of the rank of the Mordell-Weil group in higher dimensions. 
We prove an explicit bound on $ {\rm rk}\MW(X/B)$, for elliptic threefolds, with singular fibers and   $\kappa(X) = 0$. 
  For  Calabi-Yau threefolds we prove the  bound 
  $\mathrm{rk}\,{\rm MW} (X/B) \le 28$, and for elliptic fourfolds we establish  under natural assumptions the %new 
  bound $\mathrm{rk}\,{\rm MW} (X/B) \le 38$.  The precise statements %and hypotheses 
 are given in   Theorem \ref{boundCY3}, Theorem \ref{thm:boundKodZero3} and Theorem \ref{thm-4foldgen}.

We present  two types of arguments with a common philosophy: the reduction to particular elliptic surfaces $Z \subset X$.   
We  prove a natural inclusion of the Mordell-Weil groups
$\MW(X/B) \subseteq \MW(Z/C)$
(Section \ref{sec:inclusion}). 
The two types of arguments provide different sufficient conditions  for the existence of  the surface $Z$. 
  Each method leads to possible generalizations  and applications in complementary directions.  
\color{black}

For Calabi–Yau threefolds and fourfolds, the existence of a bound on ${\rm rk}\MW(X/B)$  follows from the boundedness of families of elliptic threefolds \cite{Filipazzi:2021dcw} (see also \cite{Gross1994}) and birational boundedness of fourfolds \cite{engel2025boundednessfiberedktrivialvarieties}.  
 However, existing boundedness results for elliptic Calabi-Yau do not provide explicit numerical bounds on the Mordell–Weil rank and the  boundedness  of other elliptic varieties with $\kappa(X)=0$  is proven under additional conditions, (for example, excluding  some product type),  \cite{Birkar_2024, DiCerboSvaldi, Filipazzi:2021dcw, Gross1994} and the  recent \cite{FilipazziXu2026boundednessellipticcalabiyau4folds}.

 Our  results  provide explicit numerical Mordell–Weil bounds and they also apply in certain product type and non birationally bounded families.

 At the same time, the highest Mordell–Weil rank currently known for elliptic Calabi-Yau threefolds is ${\rm rk}\MW(X/B) = 10$ 
 \cite{Grassi:2021wii,Elkies}. Also, for  Calabi-Yau manifolds, results from physics give the bound $\mathrm{rk}\,{\rm MW} (X/B)\le 28$ for threefolds and provide partial bounds for fourfolds under additional hypotheses \cite{Lee:2019skh,Martucci:2022krl,Lee:2022swr}, see Appendix \ref{sec:Appendix}.

Our theorems  prove (for Calabi-Yau threefolds)  and go beyond (for Calabi-Yau fourfolds) the  results obtained in  physics  \cite{Lee:2019skh,Martucci:2022krl,Lee:2022swr}.

Our methods apply beyond the Calabi–Yau setting. In particular, we establish explicit uniform Mordell–Weil bounds for elliptic threefolds of Kodaira dimension zero, even when the underlying varieties do not belong to birationally bounded families (see \cite{DiCerboSvaldi} and \cite{Birkar_2024}).
We improve the bounds to the known rank of $10$ when the Mordell-Weil group has non trivial torsion and $B$ is ruled.

The first method  of reduction to a surface (Section \ref{sec:arithmproof}) is via function field techniques,  and the base curve of the surface is a  general smooth movable rational curve without base points. This approach uses only elements of the classical minimal model theory for rational surfaces, without
considering degenerate fibers, and avoids difficult results in higher dimensional birational geometry. It is formulated in the scheme language. In its current state this method does not provide an explicit effective bound for Calabi-Yau fourfolds.

The second  method (Section \ref{sec:geomproof}) assumes that there are singular fibers, builds on a birational geometric approach, and  on the properties of the Shioda pairing and its height. The base curve of the fibration is not necessarily movable nor rational, and it provides  explicit effective bounds for  Calabi-Yau fourfolds.

\subsection{Outline of the  paper}
 
 In Section \ref{sec:Background} we collect the necessary background for our analysis, including results from the Minimal Model Problem applied to elliptic fibrations and properties of the Shioda pairing for rational sections. We also revisit Noether's formula (in all characteristics).

In Section \ref{sec:arithmproof}, we prove the structure theorem \ref{main2} via field extensions. 
While philosophically similar to the procedure of the following Section \ref{sec:geomproof}, the proof presented in this section differs in two aspects. It uses that the Mordell-Weil rank of an elliptic curve over a field only increases
on base change to a larger field, and that an elliptic curve over a function field of higher transcendence degree may sometimes be viewed as the generic fiber of an elliptic fibration on a smooth proper surface $S_K$, which is itself defined over some function field $K$. Then Noether's formula directly yields a bound for the Picard rank of this surface (and hence our desired elliptic Mordell-Weil rank), without considering degenerate fibers of the elliptic fibration, the Shioda pairing, or heights.  
The method in Section \ref{sec:arithmproof} is formulated in the scheme language. 
  In its current state, it does not provide an explicit effective bound for fourfolds, which we obtain, instead, by the method developed in Section \ref{sec:geomproof}. 

 In Section \ref{sec:geomproof} we combine methods from birational geometry with fundamental properties of the Shioda map and its height-pairing to  develop and prove the two structure theorems \ref{thm:BoundsFromNoetherMov} and \ref{thm:BoundsFromNoetherPic1} for elliptic fibrations (not necessarily Calabi-Yau).  
 The argument builds on  
 the physics approach of \cite{Lee:2018ihr, Lee:2019skh, Martucci:2022krl, Lee:2022swr} and proves an injection from $\MW(X/B)$ to ${\rm MW}(Z/C)$ modulo torsion  for a suitable surface $Z$ (smooth, or  birationally  either a rational or  K3 surface) obtained by restricting the elliptic fibration to a smooth  movable curve $C$ in the smooth locus of  $B$. 
 If $C$ is not movable, other arguments are needed (see in particular Lemma \ref{lemma:bsaZ}).

In Section \ref{sec:inclusion} we prove  the inclusion of the whole Mordell-Weil groups, $\MW(X/B) \subseteq \MW (Z/C)$, given by the natural restriction of sections.
 We  give sufficient conditions on  the curves $C$ for the inclusion to hold.
 
 In Section \ref{sec:dim3}, we apply  Theorems \ref{main2} and  \ref{thm:BoundsFromNoetherMov} to prove a universal bound, formulated in Theorems \ref{thm:boundKodZero3},  \ref{boundCY3} and \ref{them:boundKodZeroBNonLDP}, on the Mordell-Weil rank of  elliptic 
 threefolds with  $\kappa(X) = 0$, in particular Calabi-Yau. We describe the singularities on $B$ in Corollary \ref{cor:basketK0}.
   
   In  Proposition \ref{prop:KodNeg} we prove explicit bounds for a broad class of  threefold elliptic fibrations with $\kappa(X) < 0$. 
  \color{black}

Building on the results of Section  \ref{sec:inclusion} and \ref{sec:dim3} we describe the possible torsion subgroups of the Mordell-Weil group of elliptic threefolds in Theorem \ref{thm:TorsionCY3}. In some cases these further constrain 
the rank of the free part. 

\color{black}
In Section \ref{sec:4folds}, we apply the structure theorems  \ref{thm:BoundsFromNoetherMov} and \ref{thm:BoundsFromNoetherPic1} to obtain an explicit bound for the Mordell-Weil rank of a broad class of elliptic  fourfolds with $K_X \equiv 0$, in particular Calabi-Yau, Theorem \ref{thm-4foldgen}.   The proof requires a more delicate argument than the threefold case, because, while rational and very free curves of  minimal degree are known to exist on rank-one Fano threefolds, it is a subtle problem to find those with minimal degree (with respect to the anticanonical divisor). We find that unlike for elliptic threefolds,  the curves giving the best bound  do not need to be  movable nor are necessarily rational. A slightly modified argument proves explicit bounds for elliptic   fourfolds fibrations with $\kappa(X) < 0$.

Our  explicit bounds on ${\rm rk} \MW(X/B)$ hold  even when  $X$ does not belong to a (birational) bounded family, and  it does not satisfy the hypotheses in  \cite{Filipazzi:2021dcw}, \cite{DiCerboSvaldi} and \cite{Birkar_2024}  which ensure boundedness of ${\rm rk} \MW(X/B)$. 
The examples in Sections \ref{sec:dim3} and \ref{sec:4folds} illustrate that if  the base $B$ is not  rationally connected, the elliptic fibration decomposes  
into an  isotrivial component and a  varying part. 
The Mordell–Weil rank depends on the non isotrivial component, and it remains uniformly bounded, 
Remark \ref{remark:producttype}. 
\color{black}

Our results suggest that the Mordell–Weil rank of an elliptic fibration should admit a linear bound depending only on the dimension of the total space, and we formulate this as a conjecture in Section \ref{conjecture}.

Appendix \ref{sec:Appendix} contains a review of related work from the physics of string theory. The original motivation of this project was  in fact to prove the physics predictions for threefolds.

\color{black}

\section*{Acknowledgements}
We thank Stefano Filipazzi, Seung-Joo Lee, Luca Martucci, Matthias Schütt, and  Sho Tanimoto for discussions. 
TW especially thanks Seung-Joo Lee for collaboration on the related work \cite{Lee:2018ihr, Lee:2019skh} and Luca Martucci and Nicolo Risso for collaboration on the related work \cite{Martucci:2022krl}.
TW thanks the University of Bologna for its hospitality. KHP wishes to thank VS and the University of Buffalo for hospitality during July-Aug 2024 when VS introduced him to the problems studied here that led to this paper.
RM and VS wish to thank the organisers of the Summer Research Institute in Algebraic Geometry held in July 2025 at CSU, where we started discussing this together. 
VS thanks the UB Research Foundation for support to attend that event.

The work of TW is supported in part by Deutsche Forschungsgemeinschaft under Germany’s Excellence Strategy EXC 2121 Quantum Universe 390833306, by Deutsche Forschungsgemeinschaft through a German-Israeli Project Cooperation (DIP) grant “Holography and the Swampland” and by Deutsche Forschungsgemeinschaft through the Collaborative Research Center 1624 “Higher Structures, Moduli Spaces and Integrability.” 
The work of AG is supported in part by 
Prin 2022BTA242 ``Geometry of algebraic structures: moduli, invariants, deformations'' of Italian MUR. AG is a member of GNSAGA of  INdAM.

\section{Background on elliptic fibrations and their Mordell-Weil group} \label{sec:Background}
 We work over the complex numbers; {unless otherwise specified, all the varieties are assumed to be normal, projective and irreducible}.

\begin{definition} \label{def-CY}
    A complex variety $X$ with canonical bundle $\omega_X$ is called Calabi-Yau if $\omega_X \simeq {\mathcal O}_X$ and $h^{i}(X,{\mathcal{O} }_X) =0$ for $0<i< {\rm dim}(X)$.
\end{definition}

\subsection{Movable curves and effective divisors}
Let $B$ be a $\mathbb Q$-factorial variety; let $B^{sm} \subset B$ be the non singular locus of $B$.
Let $N_1(B)_{\mathbb{R}}$ be the set of  numerical equivalence  classes of  real $1$-cycles and 
  $N^1(B)_{\mathbb{R}}$ the set of  numerical equivalence  classes of  real  divisors.
\begin{definition}\label{movable}
\begin{enumerate}
\item  A curve $C\subset B$ is movable if and only if there is an irreducible subvariety ${\mathcal C}\subset B\times T$, with $T$ irreducible, such that all fibers $C_t,\;t\in T$, of the projection ${\mathcal C}\to T$
are irreducible curves in $B$, $C=C_{t_0}$ for some $t_0\in T$ and  the projection ${\mathcal C}\to B$ has dense image.
\item With notation as above, if, in addition, the intersection
    $\cap_{t\in T}C_t$ is empty, we say that the curve $C\subset B$ is
        freely movable.
   \item $\overline{\mathrm{NM}(B)}=\overline{\{ \sum a_i \gamma_i, \ a_i \geq 0, \ \gamma_i \in \ N_1(B)_{\mathbb{R}}, \ \gamma_i \text{ movable }\}}.$
	\item $\overline{\mathrm{Eff}(B)}=\overline{\{ \sum a_i D_i, \ a_i \geq 0, \ D_i \in \ N^1(B)_{\mathbb{R}}, \ D_i \text{ effective }\}}.$
\end{enumerate}

\end{definition}
\begin{theorem}\label{thm:BDPaP}%[Bauksom-Demailly-Paun-Peternell]
\cite{BDPP}
 Let $B$ be projective  and irreducible. Then: $$\overline{\mathrm{NM}(B)}= \{ \gamma \in N_1(B)_{\mathbb{R}} \text{ such that } D \cdot \gamma \geq 0, \ \forall  \ D \in \overline{\mathrm{Eff}(B)} \}.$$

\end{theorem}

 \begin{definition}\label{def:freeVerFree}[See for example, \cite[Section 2.4]{Debarre2011GAEL}] Let $f: \PP^1\to B$ be a non constant immersion such that $f(\PP^1) \subset B^{sm} \subseteq B$. 
 
 $C=f(\mathbb{P}^1)$  is said to be free (very free) if \( f^*T_B\) is nef (\( f^*T_B \) is ample), or equivalently
$N_{C/B}$ is the direct sum of non-negative (positive) summands.
 
 \end{definition}
In the following we denote by
 $\rho(B):=\operatorname{rk} \operatorname{Pic} (B)$.
\subsection{ Elliptic fibrations }

\begin{definition} An elliptic   fibration is a morphism $\pi : Y\rightarrow V$ whose fibers over a dense set in $V$ are elliptic 
	curves. The complement of this dense locus in $V$, the discriminant of the fibration, is  denoted by  $\Sigma_{Y/V}$.
\end{definition}

\begin{remark}
For $X$ a variety, we indicate by $K_X$ the canonical divisor, interpreted as the first Chern class of the canonical bundle $\omega_X$, and whenever a divisor $\Lambda$ is associated with a  line bundle, we call the latter $\mathcal{L}$.
\end{remark}
\begin{definition}\label{def:bireq} Let  $\tilde \pi : Y \rightarrow V$ be an elliptic fibration. A birationally equivalent elliptic fibration $  \pi :  X  \to B $ is a fibration which makes the diagram 
\[
\begin{tikzcd}
Y \arrow[d, "\tilde \pi "'] \arrow[r, dashed,"\phi"] & X \arrow[d, "\pi"] \\
V \arrow[r, dashed,"\psi" ] & B
\end{tikzcd}
\]
 commutative, where $\phi$   and $\psi$ are birational maps.
\end{definition}

\begin{defthm}[after \cite{Kodaira1963}] Let  $\tilde \pi : Y \rightarrow V$ be an elliptic fibration between smooth varieties. To the general point of each irreducible (reduced) codimension $1$ component $\Sigma_i$ of   the discriminant  $\Sigma_{Y/V}$, there is an associated $a_i \in \mathbb{Q}$.
 $\Lambda_{Y/V}\stackrel{def}=\sum _ia_i\Sigma_i$  is  %called
 the discriminant $\mathbb{Q}$-divisor.
\end{defthm}
\color{black}
\begin{theorem}\label{thm:birminmodel} 
	Let  $ 
	Y \rightarrow V$ be an elliptic fibration between smooth varieties, $\Lambda_V$ be the discriminant $\Q$-divisor. Then  the following properties hold:
    	\begin{enumerate}
        \item $\kappa(Y)=\kappa(V + \Lambda_V).$
 \item If $\Lambda_V \neq 0$, then $h^1(Y, \mathcal{O}_Y )= h^1(V, \mathcal{O}_V)$.
 
\item  If either $\dim (Y) \leq 5$ or $\dim (Y) > 5$ and klt flips exist and terminate in ${\dim Y-1}$, there exists a birationally equivalent elliptic fibration $  \pi :  X  \to B $, $X$ with  $\mathbb{Q}$-factorial terminal singularities, $( B ,   \Lambda )$ with $\mathbb{Q}$-factorial klt singularities,  such that 
			$  \Lambda  $ is supported on $\Sigma_{X/B}$, the discriminant  of the fibration, and $K_{ X }  \equiv  {\pi}^*(K_{ B } +   \Lambda )$.
  \item  If $\kappa (Y) \geq 0$ and either  $\dim (Y) \leq 5$ or $\dim (Y) > 5$ and klt flips exist and terminate in ${\dim Y-1}$, $X$ can be taken to be minimal.
\item  If $\dim (Y) \leq 5$, $\pi$ can be assumed to be
		equidimensional over an open set $U \subset B$ with $\operatorname{codim} (B \setminus U )\geq 3$.

\end{enumerate}
Furthermore:
\begin{enumerate}
 \setcounter{enumi}{5}

\item If  $\dim Y = 3$,  there are no multiple fibers, $\kappa(X)=0$  and either $\omega_X \simeq \mathcal{O}_X$   or $h^1(\mathcal{O}_X) \geq 1$, then $ B $ is smooth.
\item If  $\dim Y = 3$, $\omega_X \simeq \mathcal{O}_X$,   $h^1(X, \mathcal{O}_X)=0$, and there are no multiple fibers,   then $B$ is either $\mathbb{P}^2$ or the blow up of the Hirzebruch surface $\mathbb{F}_n, \  n \leq 12$.
\item If  $\kappa (Y)= 0$, $\Lambda \neq 0$ and either $\dim (Y) \leq 5$ or $\dim (Y) > 5$  and klt flips exist and terminate in ${\dim Y-1}$, then $B$ is birationally a Mori fiber space (see Definition \ref{def:MoriFS}).

		\end{enumerate}
	\end{theorem}
\begin{proof} (2), (6), (7) are  proved in  \cite{Grassi1991} and \cite{GrassiSingBase1993}; the other statements  are found in  \cite{GrassiWen}, which generalizes to higher dimensions the results for threefolds in \cite{Grassi1991, GrassiEqui}. \end{proof}
\begin{remark}{Note that the Minimal Model Program is known to hold  in general when  $\dim Y \leq 4$  ($Y$ is elliptically fibered and $K_Y$ cannot be  a big divisor).}
\end{remark}

\begin{definition}\label{def:MoriFS} A {Mori fiber space} is a  dominant  projective morphism
$
f : B' \to U$
such that
\begin{enumerate}[(\alph*)]
\item $B'$ is a normal projective variety with $\mathbb{Q}$-factorial terminal (klt) singularities,
\item $ f_*(\mathcal O_{B'})= \mathcal O_U$ and $\dim B '> \dim U$,
\item $ \rho(B')- \rho(U)=1$,
\item  $-K_{B'}$ is $f$-ample, that is the general fiber is a Fano variety.
\end{enumerate}
\end{definition}
The general fiber F also has ${\mathbb Q}$-factorial terminal (klt)
singularities, since $B'$ does.
By the classical results on the Minimal Model Program for threefolds, see for example \cite{MatsukiIntroMMP}, we have the following more precise structure theorem:

\begin{theorem}\label{thm:MMP} Let $\pi: X \to B$ an elliptic fibration as in Theorem \ref{thm:birminmodel},  ${\rm dim}(X) \leq 5$, and  $K_X= \pi^*(K_B+ \Lambda)$. Assume also  that $ \Lambda \neq 0$. Then there is a birational transformation $\mu: B \dasharrow B'$, where $B'$ 
has the structure of a Mori fiber space. 

Furthermore, if $B$ has terminal singularities, the following possibilities can occur:
\begin{enumerate}
\item   Any component of a fiber of $f$ is a smooth rational curve. The general fiber is a smooth rational curve.
\item The general fiber of $f$ is a smooth del Pezzo surface.
\item  $B'$ is a Fano  variety with $\mathbb{Q}$-factorial terminal singularities, $\rho (B')=1$.
\item If $\dim B'=4$, the general fiber $F$ is a Fano threefold with $\mathbb{Q}$-factorial terminal singularities.
\end{enumerate}    
\end{theorem}

\begin{definition}\label{def:RC} Let $B$ be a projective variety. $B$ is rationally connected if and only if  for  any two  very general points in $B$ there is an irreducible rational curve which contains both points.
\end{definition}

\begin{definition}\label{def:MRC} Let $B$ be a projective variety. A maximal rationally connected (MRC) fibration of $B$ is a rational map
$\phi : B \dashrightarrow Z$
such that the general fiber is rationally connected and the base $Z$ is not uniruled.
\end{definition}

 A MRC as in Definition \ref{def:MRC} exists and it is unique up to birational equivalence \cite{Campana1992,KollarMiyaokaMori1992}.

\color{black}
\medskip
\subsection{Noether's Formula}

\begin{proposition}\label{prop:Noether} Let
$
\pi \colon Z \to C,
$
where  $Z$ is a smooth  proper surface
over an algebraically closed field of characteristic 0,
$C$ is a smooth curve of genus $g(C)$ and $\omega_Z= \pi^*(\omega_C
\otimes  \mathcal{L}_C)$, for some line bundle $\mathcal{L}_C$
of degree $>0$.
Then
   \begin{equation}\label{rhoform}
\rho = \mathrm{rk}\, {\rm NS}(Z)\le 10   \deg {\mathcal L}_C+2g(C).
\end{equation}
In particular, if
$\pi: X \to B$ is as in Theorem \ref{thm:birminmodel},
$\mathcal{L}$  is  a line  bundle on $B$,  $C \subset B$ a smooth projective
curve  and
$Z=\pi^{-1}(C)$, so that
$\deg \mathcal{L}_C=\mathcal{L} \cdot C>0$, then
 \begin{equation}
\rho = \mathrm{rk}\, {\rm NS}(Z)\le 10  \mathcal{L} \cdot C+2g(C).
\end{equation}

\end{proposition}

\begin{proof}Let $Z$ be a smooth proper surface over any algebraically closed field (of any characteristic).  Recall from  \cite{PieneNoether} that Noether’s formula also applies in this case:
\begin{equation} \label{chitop1}
    \chi_{\mathrm{top}}(Z)=12\,\chi(\mathcal O_Z)-K_Z^2,
\end{equation}
where $
\chi(\mathcal O_Z)=1-h^1(Z,\mathcal O_Z) +p_g$ and
$p_g=h^{2,0}(Z)$.

On the other hand, in characteristic zero, even if the field is not the complex numbers \cite[Theorem 5.5]{Deligne1968}:
\begin{equation} \label{chitop2}
\chi_{\mathrm{top}}(Z)=\sum (-1)^i b_i (Z)=2-4\,h^{0,1}(Z)+2p_g+h^{1,1}(Z).
\end{equation}
By combining  (\ref{chitop1}) and (\ref{chitop2}) we have:
\begin{equation} \label{eq:h11Z1}
h^{1,1}(Z)=10-8\,h^{0,1}(Z)+10p_g-K_Z^2.
\end{equation}

\medskip

Now assume that $Z$ admits an elliptic fibration
\[
\pi \colon Z \to C,
\]
where $C$ is a smooth curve and $\omega_Z= \pi^*(\omega_C \otimes \mathcal{L}_C)$.
 By Riemann--Roch and using 
\[
p_g = h^0(Z,\omega_Z)=h^0(C,\omega_C \otimes\mathcal{L}_C) 
\] one finds
\begin{eqnarray} \label{pgform}
p_g&=&h^1(C,\omega_C\otimes\mathcal{L})+2g-2+\deg  \mathcal{L}_C + 1-g(C) \nonumber \\ &=& 
 h^0(C,\mathcal{L}^{-1}_C)+g(C)-1+ \deg {\mathcal L}_C   
= g(C)-1+\deg  \mathcal{L}_C \,
\end{eqnarray}
(since $h^0(C,{\mathcal L}_C^{-1})=0$, as ${\mathcal L}_C^{-1}$ has
negative degree).

{With (\ref{pgform}) and observing that $K_Z^2 =0$, we deduce from \eqref{eq:h11Z1} that}
\begin{eqnarray}
h^{1,1}(Z)
&=&10+10p_g-8h^1(Z,\mathcal O_Z)
=10+10[g(C)-1+ \deg {\mathcal L}_C  ]-8g(C) \nonumber \\
&=&10  \deg {\mathcal L}_C  +2g(C) \,. \nonumber
\end{eqnarray}
With $\mathrm{rk}\, {\rm NS}(Z) \leq h^{1,1}(Z)$ this yields (\ref{rhoform}). \end{proof}
\smallskip

A slightly modified proof of Proposition \ref{prop:Noether} gives the following weaker bounds:
\begin{proposition}\label{prop:NoetherAnyC} Let
$
\pi \colon Z \to C,
$
where  $Z$ is a smooth  proper surface
over an algebraically closed field of any characteristic,
$C$ is a smooth curve of genus $g(C)$ and $\omega_Z= \pi^*(\omega_C
\otimes  \mathcal{L}_C)$, for some line bundle $\mathcal{L}_C$
of degree $>0$.
Then
  \begin{equation*}\label{rhoform2}
\rho (Z) \le b_2 (Z)\le 12 \deg {\mathcal L}_C + 4 g(C)  -2
\end{equation*}
and 
 \begin{equation*}
\rho (Z) \le b_2 (Z) \le 12  \mathcal{L} \cdot C+4g(C)-2 \,,
\end{equation*}   
respectively.
\end{proposition}

\begin{comment}

\end{comment}
\vskip 0.4in

\subsection{The Shioda map and the Shioda pairing  for elliptic fibrations with sections}~

\medskip
 From now on   we will consider elliptic fibrations with a rational section.

Let $\MW(X/B)$ be the  Mordell-Weil group of the elliptic fibration $\pi: X\to B$,  defined as  the 
 group of rational sections of
the fibration.
  Equivalently, we can write the Mordell-Weil group as $E(K)$ for $K=K(B)$ the function field of $B$.  The zero section (the zero element of the group) will be denoted by $S_0$ (or, if the zero section is rational, by the closure of its image). In particular the existence of a zero section implies that there are no multiple fibers.
   Also, there is a canonical
identification of groups between ${\rm Pic}^0(E_K)$ and the Mordell-Weil
group $E(K)$.

\begin{proposition} Let $Y \to V$ and $ X \to B$ be birationally equivalent fibrations as in Definition \ref{def:bireq}. Then $\MW(Y/V) \simeq \MW(X/B) $.

\end{proposition}
\begin{proof}For two birationally equivalent fibrations as in Definition \ref{def:bireq}, the
generic elliptic fibers over the (common)
function field of the two bases are isomorphic, and therefore have the same
Mordell-Weil groups. \end{proof}
 Therefore, to derive bounds on the rank of the Mordell-Weil group, we can, without loss of generality, assume that 
$X$, $B$ and  $\pi: X\to B$ are as in Theorem \ref{thm:birminmodel}.

 \begin{proposition}\label{prop:MWtorinj} If  $C \subset B$ is a  general smooth movable curve, 
 {and $Z$ denotes the elliptic surface obtained by restricting the elliptic fibration of $X$ to $C$,}
then
 $\MW(X/B)_{\rm tors} \subseteq \MW(Z/C)_{\rm tors} .$
\end{proposition}
\begin{proof} The locus where two different torsion sections coincide is contained in a proper divisor of $B$. A  sufficiently general member of a movable family curve $C$ is not contained in this locus. \end{proof}
\color{black}

 From now on,  we assume that the rank of $\MW(X/B)$ is non-vanishing  so that  there are non-torsion rational sections, i.e.
 ${\rm rk}\,  \MW(X/B) \not=0$.

\begin{definition} \label{def:shioda}\cite{Shioda1990, Morrison:2012ei}
  For an elliptic fibration $\pi:X\to B$
as above, if $\dim B \leq 3$, there is a Shioda map from the Mordell-Weil group to the Neron-Severi group $\NS(X)\otimes\mathbb{Q}$,
  \[
    \sigma \;:\; \MW(X/B)\longrightarrow \NS(X)\otimes\mathbb{Q} \,,
  \]
  which associates to a section $S_a$ the class
  \begin{equation}\label{eq:shioda-general}
    \sigma(S_a) \;=\; S_a - S_0
    - \pi^*\!\bigl(\pi_*( (S_a - S_0)\cdot S_0 )\bigr)
    + \sum_{i}\sum_{A} \ell^{\,i}_A(S_a)\,E_i^{A}.
  \end{equation}
Here
  \begin{itemize}
    \item $S_a$ denotes the divisor (section) corresponding to the section  $a$  (or,  if the section is rational, the closure of its image);
    \item $S_0$ is the zero section;
    \item $E_i^{A}$ is a rationally fibered divisor  such that $\pi(E_i^A) $  is an irreducible component $\Sigma_A \subseteq\Sigma_{X/B}$ and $E_i^A$ is the $i$-th irreducible component of  $\pi^{-1}(\Sigma_A)$;
     \item $\ell^{\,i}_A(S_a)$ are rational coefficients which measure how $S_a$ meets the reducible fibers.
  \end{itemize}
\end{definition}

\begin{remark}\label{remark:ell-vanishing}
\begin{enumerate}[(\alph*)]
\item  The Shioda map is  a group homomorphism. 
\item The term $\ell^{\,i}_A(S_a)$ contains as a factor the intersection multiplicity of $S_a$ with the general fibre of the 
 divisors $E^A_i$.
  In particular, if $S_a$ intersects the singular fibres in the same component as the zero-section over the generic point of the discriminant, then all $\ell^{\,i}_A(S_a)=0$.

 \item The restriction on
$\dim B$ is not needed in Definition \ref{def:shioda}, though this is not
explicitly done in the literature. This can be deduced from the
Hard Lefschetz theorem and the Hodge index theorem (%basically that
the
pure Hodge structure $\NS(X)\otimes{\mathbb Q}$ carries a polarization,
determined using Hodge theory).

\item  Theorem \ref{thm:birminmodel} (1) implies that there is a surjection
\[{\rm NS}(X)\otimes{\mathbb Q} \twoheadrightarrow {\rm Pic}(E_K)\otimes{\mathbb
Q}=\MW(X/B)\otimes{\mathbb Q}\oplus {\mathbb Q},\]
and the kernel of the induced projection onto $\MW(X/B)\otimes{\mathbb
Q}$ is a non-degenerate subspace for the polarization, because of the
Hodge-Riemann bilinear relations (which is the Hodge index theorem,
for $\NS(X)$), hence is a split surjection of rational vector spaces.
This also accounts for the fact that the Shioda map is a group
homomorphism, whose kernel is the torsion subgroup.

 \item 
 Finally, the Shioda map is injective after tensoring with ${\mathbb Q}$.

\end{enumerate}

\end{remark}

\begin{definition}\label{def:heightpairing}
  For a section $S_a$ we define the class
  \[
    b_{S_a}
    := -\pi_*   \bigl( \sigma(S_a) \cdot  \sigma(S_a)  \bigr) = \langle  \sigma(S_a),  \sigma(S_a) \rangle  \,,
  \]
   which is viewed as an element of
$N^1(B)_{\mathbb R}$. This is the usual height pairing associated to
the section, if $B$ is a smooth projective curve, so that
$N^1(B)_{\mathbb R}={\mathbb R}$. 
\end{definition}

\begin{proposition}\label{prop:CoxZucker}
  Let $S_a$ be any (rational) section distinct from the zero section $S_0$.  Then there  exist an open set $\mathcal{V} \subset B$  and  an integer $m\in\mathbb{N}$ such that the multiple $m S_a$ (viewed in the Mordell-Weil group) intersects  the singular fibres in the same component as $S_0$ over $\mathcal{V}$.
\end{proposition}

\begin{proof} Let $\mathcal{V} \subset B$ be the open set where $S_a$ is regular. 
  The statement follows from an extension of the proof of Lemma (1.5) in \cite{CoxZucker1979} (or standard results on the behaviour of sections under multiplication by integers). 
   (See  \cite{Lee:2018ihr} for details.)  Over $\mathcal{V}$  the multiple $mS_a$  does not intersect the non-identity components of singular fibers for suitable $m$, and hence lies in the same component as $S_0$ there. \end{proof}

\begin{remark}
We note
that the smooth locus of any fibre of $\pi$ has a group structure,
which is a semidirect product of the subgroup given by the
0-component, and a finite abelian group (the group of connected
components). Proposition \ref{prop:CoxZucker} is a consequence.
\end{remark}

\begin{proposition}\label{prop:mSa}
    Let $S_a$ be any section distinct from the zero section $S_0$.  Then there exists an integer $m\in\mathbb{N}$ such that   
	\begin{equation}\label{eq:shioda-general2}
		m\sigma(S_a) \;=\; mS_a - S_0
		- \pi^*\!\bigl(\pi_*( (mS_a - S_0)\cdot S_0 )\bigr).
	\end{equation}
\end{proposition}

\begin{proof} The statement follows from  Proposition \ref{prop:CoxZucker} and Remark \ref{remark:ell-vanishing}. \end{proof}
	
\begin{proposition}\label{prop:fmw}
Assume $\omega_X \simeq \pi^*(\omega_B \otimes \mathcal{L})$, with $\mathcal{L}$ a line bundle on $B$ with associated divisor $\Lambda$. Then for all $S_a$, $-\pi_\ast(S_a \cdot S_a) = \Lambda$.
\end{proposition}
\begin{proof}
The statement follows from an argument in \cite[(7.30)]{FMW1997}.\end{proof}

\begin{corollary}\label{cor:fmw}
\label{prop:kx-sigma}
  Assume  that $\omega_X \simeq \mathcal{O}_X$.  Then for every section $S_a$ we have
  \[
    -\pi_*\bigl( S_a \cdot S_a\bigr) \;=\; -K_B.
  \]
\end{corollary}

\begin{proposition}\label{prop:bsa}Assume $\omega_X \simeq \pi^*(\omega_B \otimes  \mathcal{L})$,  with $\mathcal{L}$ a line bundle on $B$ with associated divisor $\Lambda$. Let $S_a$ be any section distinct from the zero section $S_0$.  Then there exists a positive integer $m$ such that
	\[
	b_{S_a}
	= \frac{1}{m^2}\Big(2 \Lambda + 2\pi_*\bigl((mS_a)\cdot S_0\bigr)\Big).
	\]
   Furthermore $b_{S_a} \in \overline{\operatorname{Eff}(B)}$.
	
\end{proposition}
\begin{proof}
  Let  $m$ be as in  Proposition \ref{prop:CoxZucker} so that $mS_a$ meets the singular fibers in the $S_0$-component over generic points. The statement follows from Propositions \ref{prop:mSa} and \ref{prop:fmw}.\end{proof}
\begin{corollary}\label{cor:bsa} Let $S_a$ be any section distinct from the zero section $S_0$.  Assume that $\omega_X \simeq \mathcal{O}_X$.  Then there exists a positive integer $m$ such that
  \[
    b_{S_a}
    = \frac{1}{m^2}\Big(-2K_B + 2\pi_*\bigl((mS_a)\cdot S_0\bigr)\Big).
  \]
Furthermore $b_{S_a} \in \overline{\operatorname{Eff}(B)}$.
\end{corollary}

Finally, we recall the following:

\begin{theorem}[Shioda-Tate-Wazir formula] \label{thmSTW} Let $X\to B $ be as in Theorem \ref{thm:birminmodel} with $\Lambda \neq 0$. Then
 the rank of the Mordell-Weil group of $X$ is bounded by the Neron-Severi group of $X$ as follows:
\begin{equation}
{\rm rk}(\MW(X/B)) \leq  {\rm rk}(\NS(X))- {\rm rk}(\NS(B)) -1 \,.
\end{equation}

\begin{proof}  Terminal and klt singularities are rational and the statement follows from \cite[Theorem 1.1]{OguisoOnShiodaWazir}, which extends \cite{Shioda1990, Wazir_2004}.\end{proof}

\end{theorem}

\section{ 
Bounds for the Mordell-Weil rank with Function Fields } \label{sec:arithmproof}

{In this section we present an arithmetic proof of the following structure theorem bounding the Mordell-Weil rank of elliptic fibrations:}
\begin{theorem}\label{main2}
Let $\pi:X\to B$ be a morphism between projective, normal complex algebraic varieties, such that \\
(i) $X$ is Cohen-Macaulay, and its singular locus has codimension at least 3;\\
(ii) $\pi$ is an elliptic fibration with a section; \\
(iii) the canonical sheaf of $X$ is trivial ($\omega_X\cong {\mathcal O}_X$).
\begin{comment}
and $h^1(X,{\mathcal O}_X)=0$.
\end{comment}

Suppose that there is a freely movable smooth rational curve $C$ in $B$ with normal bundle $\mathcal{N}_{C}$ of degree $d$.
 Then the elliptic curve $E_F$ over the function field $F={\mathbb C}(B)$
has Mordell-Weil group $E_F(F)$ with
\begin{equation}
       {\rm rk}E_F(F) \leq 18+10d = 10 (-C) \cdot K_B -2 \,,
\end{equation}
where $d$ is the degree of the normal bundle of $C_t$ for any $t\in T$.
\end{theorem}

\begin{remark}
From Definition~\ref{movable} recall that $C$ being a freely movable smooth rational curve in $B$ has the following implications:
\begin{itemize}
    \item There is an (irreducible) nonsingular complex variety $T$.
    \item There is a smooth morphism
$\varphi:T\times{\mathbb P}^1_{\mathbb C}\to B$ with dense image,
contained in the smooth locus of $B$.
    \item There is a $t_0\in T$ such that $C_{t_0}=C$.
    \item Regarding this as a family of morphisms $\varphi_t:{\mathbb
        P}^1_{\mathbb C}\to B$, we may shrink $T$ to assume that $\varphi_t$ is an embedding
        for every $t$ in $T$ with image $C_t$.
    \item The family $\{C_t\}_{t\in T}$ has no base points
({\it i.e.}, for any $b\in B$ there exists $t\in T$
such that the corresponding rational curve $C_t$ does not contain $b$).
\end{itemize}
\end{remark}
\color{black}

\begin{remark}
Because of our hypothesis that the family of rational curves $C_t$ has no base points,
the normal bundle of any such $C_t$ must be generated by its global sections,
and so $d\geq 0$.
\end{remark}

\begin{remark} 
Movable curves are not necessarily base point free. However, the movable curves  we consider in Section \ref{sec:dim3} are base point free.
\end{remark}

\subsection{Proof of Theorem \protect{\ref{main2}}}

Recall that for an elliptic curve $E$ over a field $K$,
which has a given rational point $P\in E(K)$,
the Mordell-Weil group $E(K)$ of $K$-rational points is naturally identified
with the group ${\rm Pic}^0(E)$ of isomorphism classes of line bundles on $E$ of degree 0,
via the map
\[
Q (\in E(K))\mapsto {\mathcal O}_E(Q-P) (\in {\rm Pic}^0(E)).
\]
Hence if $K\subset L$ is any extension of fields,
then ${\rm Pic}^0(E)\to {\rm Pic}^0(E_L)$ is injective,
since $E(K)\subset E_L(L)$ in an obvious way.

With the hypotheses and notation of Theorem \ref{main2},
consider the ``incidence variety''
\[
Z=(T\times{\mathbb P}^1_{\mathbb C})\times_BX:=
\{(t,u,x)\in T\times{\mathbb P}^1_{\mathbb C}\times X\mid \varphi(t,u)=\pi(x)\} \,.
\]
The image $\pi(X_{sing})$ in $B$ of the singular locus of $X$ has codimension at least 2 in $B$,
since $X_{sing}$ has codimension at least 3, by assumption.
Since $\varphi$ is smooth,
$\varphi^{-1}(\pi(X_{sing}))$ is also of codimension at least $2$
in $T\times{\mathbb P}^1_{\mathbb C}$.
Hence, if we replace $T$ by a nonempty Zariski open subset,
we may assume $\varphi(T\times{\mathbb P}^1_{\mathbb C})$
does not intersect  $\pi(X_{sing})$.
Now since $\varphi$ is a smooth morphism,
so is the morphism $f:Z\to X$ induced by the projection,
and its image is disjoint from $X_{sing}$.
Hence $Z$ is nonsingular. It is also clear, since it is a pull-back,
that the morphism $Z\to T\times{\mathbb P}^1_{\mathbb C}$
is an elliptic fibration with a cross section, since $X\to B$ has this structure.

The generic fibre of $f:Z\to T\times{\mathbb P}^1$ is $E_K=E\times_F K$,
where $F={\mathbb C}(B)$ is the function field of $B$,
$K={\mathbb C}(T\times{\mathbb P}^1_{\mathbb C})$
is the function field of $T\times{\mathbb P}^1_{\mathbb C}$,
and $F\hookrightarrow K$ is the inclusion of fields induced by $\varphi$.
Hence it suffices to prove the stated bounds on the Mordell-Weil rank for the elliptic curve $E_K$, since that rank can only increase after a base change of fields.

We next note that, if $L={\mathbb C}(T)$ is the function field of $T$,
we may consider the elliptic surface $Z_L\to {\mathbb P}^1_L$
obtained as the generic fiber of the composition
\[
g:Z\to T\times{\mathbb P}^1_{\mathbb C} \to T
\]
where the second map is the projection to $T$.
Then $Z_L\to {\mathbb P}^1_L$ is a nonsingular projective surface
defined over the field $L$, which has an elliptic fibration
(the map to ${\mathbb P}^1_L$) with a section,
and whose generic fiber is the {\em same} elliptic curve $E_K$.

We also note that for a Zariski open dense subset of $t\in T$,
the fiber $Z_t=g^{-1}(t)=\pi^{-1}(C_t)\subset X$ is a nonsingular projective complex surface.

Consider an algebraic closure $\bar{L}$ of $L$, and the corresponding surface $Z_{\bar{L}}$.  This surface comes equipped with an elliptic fibration
\[
h:Z_{\bar{L}}\to {\mathbb P}^1_{\bar{L}}
\]
with a cross section, and if $M=\bar{L}({\mathbb P}^1)$ is the function field,
then there are field extensions
\[
K\hookrightarrow L \hookrightarrow \bar{L}\hookrightarrow M,
\]
which allows us to identify the generic fiber of $h$ with the elliptic curve
\[
E_M=E_F\times_FM,
\]
so that
\[
E(F)\subset E(M)
\]
is a subgroup.
It clearly suffices to prove the bounds stated in Theorem~\ref{main2}
for this ``new'' elliptic curve $E_M$,
which is the generic fiber of $h:Z_{\bar{L}}\to {\mathbb P}^1_{\bar{L}}$.

By the following lemma $H^1(Z_{\bar{L}},{\mathcal O}_{Z_{\bar{L}}})$
vanishes. \color{black}%
Hence the Picard group is discrete, and equal to its Neron-Severi group;
by application of Theorem \ref{thmSTW} we therefore have the estimate
\[
{\rm rk}\,E(M)\leq {\rm rk}\,\NS(Z_{\bar{L}})-2 \,.
\]
 The bound stated in Theorem~\ref{main2} follows immediately from the lemma below.

\begin{lemma} \label{eq:rkNS-1}
Under the hypotheses of Theorem~\ref{main2}, and with the above notations,
if $d$ is the degree of the normal bundle of $C_t\subset B$, where $t\in T$,
then 
$h^1(Z_{\bar{L}},{\mathcal O}_{Z_{\bar{L}}})=0$ and
\color{black}
\[
{\rm rk}(\NS(Z_{\bar L})) \leq 20 +10d \,.
\]
\end{lemma}
\begin{proof}
We note that if $t\in T$, and $Z_t=\pi^{-1}(C_t)\subset X$
is such that $Z_t$ is a nonsingular complex surface,
with its induced elliptic fibration $\pi_t:Z_t\to{\mathbb P}^1_{\mathbb C}$,
then (using adjunction on $X$, which has trivial canonical bundle) we see that
$Z_t$ has canonical line bundle
\[
\omega_{Z_t}={\rm det}{\mathcal N_{Z_t}} = \pi_t^*({\rm det}{\mathcal N}_{C_t}),
\]
where ${\mathcal N}_{Z_t}$ is the normal bundle of $Z_t\subset X$,
and ${\mathcal N}_{C_t}$ is the normal bundle of $C_t$ in $B$.
Since $C_t\cong {\mathbb P}^1_{\mathbb C}$,
we have that ${\rm det}{\mathcal N}_{C_t}\cong {\mathcal O}_{{\mathbb P}^1}(d)$
where $d$ is the degree of the normal bundle of $C_t$.
Similar reasoning shows that
\[
\omega_{Z_{\bar{L}}}\cong h^*{\mathcal O}_{{\mathbb P}^1_{\bar{L}}}(d).
\]

Using \cite[Corollary 2.4(iv)]{M} we see that
    $R^1h_*\mathcal{O}_{Z_{\bar{L}}}=\mathcal{O}_{\mathbb{P}^1}(-d-2)$.
By the Leray spectral sequence for $h$, the vector space
$H^1(Z_{\bar{L}},{\mathcal O}_{Z_{\bar{L}}})$ is trapped between
    $H^1(\mathbb{P}^1,\mathcal{O}_{\mathbb{P}^1})=0$ and
    $H^0(\mathbb{P}^1,\mathcal{O}_{\mathbb{P}^1}(-d-2))=0$; so it
    vanishes.
\color{black}

Hence the surface $Z_{\bar{L}}$ has geometric genus
\[
p_g(Z_{\bar{L}})=\dim H^0(Z_{\bar{L}},\omega_{Z_{\bar{L}}})=d+1,
\]
and
\[
q(Z_{\bar{L}})=\dim H^1(Z_{\bar{L}},{\mathcal O}_{Z_{\bar{L}}})=0.
\]
Noether's formula  in Proposition \ref{prop:Noether} then implies that
\[
{\rm rk} (\NS(Z_{\bar{L}}))\leq 12(1+p_g)-2-2p_g=12(d+2)-2-2(d+1)=10d+20.
\]
\end{proof}

The varieties $X$, $B$, $T$ and the morphisms $\pi$ and $\phi$ are
defined over a countable algebraically closed subfield $k\subset \mathbb{C}$.
The proof given above can be carried out over $k$ instead of
$\mathbb{C}$.

\smallbreak

\begin{remark}\begin{enumerate}[(i)]
\item Recall from Proposition \ref{prop:Noether} that Noether's formula
is also valid for a smooth projective surface over an algebraically closed field,
when interpreted suitably \cite{PieneNoether};  thus the resulting inequality on the rank of the Neron-Severi group is valid in characteristic 0 as given in  Lemma \ref{eq:rkNS-1} and a slightly weaker upper bound holds in any characteristic.
Since $\bar{L}$ has characteristic $0$,
this can in fact be deduced from the case of complex surfaces,
together with the assertion that for a surface $S$ over an algebraically closed field $K$,
and an extension of algebraically closed fields $K\subset L$,
the Neron-Severi groups of $S$ and of $S_L$ are canonically isomorphic.
\item  
We recall that our Mordell-Weil group $E_F(F)$ corresponds to
the group of {\em rational sections} of the elliptic fibration
$\pi:X\to B$, since it clearly depends only on the generic fiber,
viewed as an elliptic curve  over the function field of the base
variety. 

\end{enumerate}
\end{remark}

\section{
Bounds for the Mordell-Weil rank via birational geometry
} \label{sec:geomproof}

\medskip

In the following we assume  $X$ as in Theorem \ref{thm:birminmodel}, 
$K_X = \pi^\ast(K_B + \Lambda)$ and $\Lambda \neq 0$.
The results in this
section hold for any dimension, unless the dimension is specified.
{In particular, we will prove the structure theorems \ref{thm:BoundsFromNoetherMov} and 
\ref{thm:BoundsFromNoetherPic1}.}

\begin{proposition}\label{prop:SectionMovableCurveInt}
	Let $S_a$ be a section.
	Then $
		b_{S_a}\cdot C \ge 0 \ 
		$
     for every curve class $C\in \overline{\mathrm{NM}(B)}$.

\end{proposition}

\begin{proof}  By Proposition \ref{prop:bsa}, $b_{S_a} \in \overline{\operatorname{Eff}(B)}$ and  the cone  $\overline{\mathrm{Eff}(B)}$ is dual to $\overline{\mathrm{NM}(B)}$, by Theorem \ref{thm:BDPaP}.\end{proof}

\begin{lemma}\label{lemma:bsaZ}
Let $S_a \in \MW(X/B)$, $ C \subset B $  be a smooth curve contained in $B^{sm}$, the smooth locus of $B$, but such that $ C \not \subset\Sigma_{X/B}$,  and $ Z \stackrel{def}= \pi^{-1}(C)$  . %smooth}. 
Then
\[
 b_{S_a} \cdot C= -{\pi_\ast}_{|\widehat Z}\bigl(\sigma ({S_a}_{|\widehat Z}) \cdot  \sigma({S_a}_{|\widehat Z}) \bigr) = {\langle \sigma ({S_a}_{|\widehat Z}) , \sigma({S_a}_{|\widehat Z})\rangle}_{|_{\widehat Z}}\,
\]
is  well-defined and gives the height pairing $\langle,\rangle_{|_{\widehat Z}}$ for $S_a|_{\widehat Z}\in \MW({\widehat Z}/C)$, where  $ \tau: \widehat Z \to Z$ is a minimal resolution  of $Z$ relative to the morphism $Z \to C$.
\end{lemma}

\begin{proof} 
Recall that $ b_{S_a}
   = -\pi_*   \bigl( \sigma(S_a) \cdot  \sigma(S_a)  \bigr) = \langle  \sigma(S_a),  \sigma(S_a) \rangle.$
  By Proposition \ref{prop:CoxZucker} there exists a positive integer $m$ such that 
$mS_a$ intersects the singular fibers, possibly after restricting to an open neighborhood,  in the same component as $S_0$ and then $ b_{S_a}
    = \frac{1}{m^2}\Big(2 \Lambda + 2\pi_*\bigl((mS_a)\cdot S_0\bigr)\Big)
  $. There exists a Weierstrass model $W_X$ of $X_{|\pi^{-1}(B^{sm})} \to B^{sm}$ and   a Weierstrass model $W_Z$ of $Z\to C$. Now the birational map between $X$ and $W_X$ is an isomorphism in an open set which includes $(mS_a)\cdot S_0$ for a suitable choice of $m$. Likewise, the birational maps between $Z$, $\widehat Z$ and $W_Z$ are isomorphisms in an open set of $((mS_a)\cdot S_0))_{|\pi^{-1}(C)}$. Then 
$\pi_*\bigl((mS_a)\cdot S_0\bigr)\bigr)\cdot C=(\pi_{|\widehat Z})_*\bigl((mS_a)_{|\widehat Z}\cdot {S_0}_{|\widehat Z}\bigr)\Big),$  where  $(mS_a)_{|\widehat Z} \ $ is also the addition in $\widehat Z$, {by base change}. We also have  $K_{\widehat Z}= \pi_{|\widehat Z}^*(K_C+ \Lambda \cdot C)$ and thus
\begin{eqnarray} b_{S_a} \cdot C
   & =& \frac{1}{m^2}\Big(2 \Lambda + 2\pi_*\bigl((mS_a)\cdot S_0\bigr)\Big)\cdot C \nonumber \\
   &=& \frac{1}{m^2}\Big(2 \Lambda \cdot C + 2(\pi_{|\widehat Z})_*\bigl((mS_a)_{|\widehat Z}\cdot {S_0}_{|\widehat Z}\bigr)\Big) \nonumber
\end{eqnarray}
and 
 \[ b_{S_a} \cdot C = -{\pi_\ast}_{|\widehat Z}\bigl(\sigma ({S_a}_{|\widehat Z}) \cdot  \sigma({S_a}_{|\widehat Z}) \bigr) = \langle \sigma ({S_a}_{|\widehat Z}) , \sigma({S_a}_{|\widehat Z})\rangle_{|_{\widehat Z}} \,.
 \]\end{proof}

\begin{remark}
The
compatibility of the height pairing with appropriate base changes, as
recorded in Lemma \ref{lemma:bsaZ}, may be viewed as a special case of a formula
in intersection theory, \cite{Fulton}, Theorem
6.2 (a).
\end{remark}

\begin{notation} \label{remark_notation} In the hypothesis of the previous Lemma \ref{lemma:bsaZ}, from now on  we will write $\MW(Z/C)$ in place of 
$\MW({\widehat Z}/C)$. In many cases  indeed $Z = \widehat Z$.
\end{notation}

\begin{proposition}\label{prop:InjMW} Let $C \subset B^{sm} \subseteq B$ be a smooth curve, $C \not \subset  \Sigma_{X/B}$ and %such that 
$\pi^{-1}(C)=Z$% is a smooth surface
. Assume  that $b_{S_a} \cdot C \geq 0$ for every $S_a$ and that
$b_{S_a} \cdot C = 0$ if and only if $b_{S_a}=0$, i.e. if $S_a$ is torsion.
Then 
\begin{enumerate}
    \item $ 
S_a|_Z \text{ is  torsion in } \MW(Z/C) \Longleftrightarrow  S_a \text{ is  torsion in } \MW(X/B), $
\item $ S_a \neq S_b \in \MW(X/B)/\MW(X/B)_{\rm tors} \Longrightarrow $\\
$S_a|_Z \neq S_b|_Z  \in \MW(Z/C)/\MW(Z/C)_{\rm tors}.
$
\end{enumerate}
  
\end{proposition}

\begin{proof}
  Recall that $
  b_{S_a}=
\langle S_a, S_a \rangle = 0$ if and only if $ S_a$  is torsion (see also Remark \ref{remark:ell-vanishing}, (d)). The first statement follows then  from  the assumptions and Lemma \ref{lemma:bsaZ}.   The same lemma and the  Shioda homomorphism  imply the second part.
\end{proof}

\begin{theorem}\label{thm:InjMW}
Let $C \subset B^{sm} \subseteq B$ be a smooth curve, $C \not \subset  \Sigma_{X/B}$. Assume  that $b_{S_a} \cdot C \geq 0$ for every $S_a$ and
that $b_{S_a} \cdot C = 0$ if and only if ${S_a}$ is a torsion section. Then
\[\MW(X/B)/( \MW(X/B))_{tor} \subseteq  \MW(Z/C)/ (\MW(Z/C))_{tor}
\]
and
\[ 
\operatorname{rk} \MW(X/B) \leq \operatorname{rk} \MW(Z/C)   \,.
\]
\end{theorem}
\begin{proof} This follows from  Proposition \ref{prop:InjMW}. 
\end{proof}

\begin{theorem}\label{thm:InjMWsmooth}
Let $C \subset B^{sm} \subseteq B$ be a smooth curve such that $\pi^{-1}(C)=Z$ is a smooth surface. Assume  that $b_{S_a} \cdot C \geq 0$ for every $S_a$ and
that $b_{S_a} \cdot C = 0$ if and only if ${S_a}$ is a torsion section. Then
\[ 
\operatorname{rk} \MW(X/B) \leq \operatorname{rk} \MW(Z/C)   \leq \operatorname{rk} \NS(Z) - 2  \,.
\]
\end{theorem}
\begin{proof} The first inequality again follows from  Proposition \ref{prop:InjMW}. The second inequality is an application of Theorem \ref{thmSTW} to the  smooth elliptic surface $Z$.
\end{proof}

\begin{proposition}\label{prop-positiveintMov} Let $C\in \overline{\mathrm{NM}(B)}$ be a %a smooth
curve% such that $\pi^{-1}(C)=Z$ is a smooth surface
.
Assume also that $\Lambda \cdot C > 0$.  Then 
$b_{S_a} \cdot C \geq 0$ for every $S_a$, and
$b_{S_a} \cdot C = 0$ if and only if %$b_{S_a}=0$, that is
$S_a$ is torsion.

\end{proposition}

\begin{proof}  Since $\Lambda \cdot C > 0$, Propositions \ref{prop:bsa}  and \ref{prop:SectionMovableCurveInt} imply that $b_{S_a} \cdot C \geq 0$ and $b_{S_a} \cdot C = 0$ if and only if $b_{S_a}=0$. \end{proof}

\begin{proposition}\label{prop-positiveintPic1}  Let  $C \subset B$  be a %smooth 
curve.
%such $\pi^{-1}(C)=Z$ is a smooth surface
 Assume also that $\Lambda \cdot C > 0$ and $\operatorname{rk} \Pic (B)=1$.  Then $b_{S_a} \cdot C \geq 0$ for every $S_a$, and
$b_{S_a} \cdot C = 0$ if and only if $b_{S_a}=0$, that is $S_a$ is torsion. 
\end{proposition}
\begin{proof} The hypotheses imply that $\overline{\mathrm{Eff}(B)}$ coincides with the closure  in $\mathbb{R}$ of the ample cone.
    \end{proof}

\begin{corollary}\label{cor:bsaNonNegMovANDPic1}  Let  $C  \subset  B$  be a smooth curve. Assume that 
$\Lambda = - K_B$. Assume also that either $C \in \overline{\mathrm{NM}(B)}$ is  rational or that  $\operatorname{rk} \Pic (B)=1 $  with $B$ Fano.
Then $b_{S_a} \cdot C \geq 0$ for every $S_a$, and
$b_{S_a} \cdot C = 0$ if and only if $S_a$ is torsion.
\end{corollary}

\begin{proof} We only need to check the first case: if $C \in \overline{\mathrm{NM}(B)}$ is smooth and rational, the adjunction formula  gives $-K_B \cdot C = 2+ \deg N_{C/B} > 0$.  The statements follow from Propositions \ref{prop-positiveintMov} and \ref{prop-positiveintPic1}.
 \end{proof}

\begin{proposition} \label{prop:K3appearance1}Let $C \subset B^{sm} \subseteq B $ be a smooth rational curve. Assume that
$b_{S_a} \cdot C \geq 0$ for every $S_a$, and
$b_{S_a} \cdot C = 0$ if and only if $b_{S_a}=0$, that is, $S_a$ is torsion. 
\begin{enumerate}
    \item  If in addition $\Lambda \cdot C=2$, then $Z=\pi^{-1}(C)$ is  (birationally) a K3 surface and
\begin{equation*} \label{eq:rank18bound}
    \operatorname{rk} \MW(X/B) \leq \operatorname{rk} \MW(Z/C)   \leq 18.
    \end{equation*}
    \item  If in addition $\Lambda \cdot C=1$, then $Z=\pi^{-1}(C)$ is  (birationally) a rational elliptic surface and
\begin{equation*} \label{eq:rank18bound}
    \operatorname{rk} \MW(X/B) \leq \operatorname{rk} \MW(Z/C)   \leq 8.
    \end{equation*}
\color{black}
\end{enumerate}

    \end{proposition}
\begin{proof} 
By construction, $Z=\pi^{-1}(C)$ is an elliptic surface with  $ 12 \Lambda \cdot C =24$  or $ 12 \Lambda \cdot C =12 $ singular fibers and  is therefore birationally a K3 surface or a rational elliptic surface. 
 The first part of the inequality 
 follows from Theorem \ref{thm:InjMW}.
   The second part of the inequality  follows from the known bound
%$\operatorname{rk} \MW(Z/C)   \leq 18$
for $Z$ a K3 surface or a rational elliptic surface (recalling Notation \ref{remark_notation}).
\end{proof}

\begin{corollary}\label{cor:MWForZK3}
 Let $C \subset B^{sm} \subseteq  B $ be a smooth rational curve
 with normal bundle $ N_{C/B} \simeq \oplus ^{\dim B -1}\mathcal{O}_C$. 
  Assume that
$b_{S_a} \cdot C \geq 0$ for every $S_a$, and
$b_{S_a} \cdot C = 0$ if and only if $b_{S_a}=0$, that is, $S_a$ is torsion. 
In addition assume that $\Lambda = - K_B$. 
Then $Z=\pi^{-1}(C)$ is birationally a K3 surface and
   \begin{equation} 
    \operatorname{rk} \MW(X/B) \leq \operatorname{rk} \MW(Z/C)   \leq 18 \,.
\end{equation}  \end{corollary}
\begin{proof}
By adjunction for the rational curve $C$,
$\Lambda \cdot C = -K_B \cdot C = 2 + {\rm deg}N_{C/B} = 2$. The claim then follows from Proposition \ref{prop:K3appearance1}.
\end{proof}

\begin{theorem}\label{thm:MWNSMov}
 Let $C\in \overline{\mathrm{NM}(B)}$ be  a smooth curve    such that   $Z = \pi^{-1}(C) $ is smooth 
 with $\Lambda \cdot C > 0$.
 Then
\[ 
\operatorname{rk} \MW(X/B) \leq \operatorname{rk} \MW(Z/C)   \leq \operatorname{rk} \NS(Z) - 2  \,.
\]
\end{theorem}

\begin{proof}  Proposition \ref{prop-positiveintMov} implies that the hypotheses of Theorem \ref{thm:InjMWsmooth} are satisfied. \end{proof}

\begin{theorem}\label{thm:MWNSPic1} 

Let  $C \subset B$  be a smooth 
curve such that $\pi^{-1}(C)=Z$ is smooth.
 Assume also that $\Lambda \cdot C > 0$ and $\operatorname{rk} \Pic (B)=1$.  Then 
\[ 
\operatorname{rk} \MW(X/B) \leq \operatorname{rk} \MW(Z/C)   \leq \operatorname{rk} \NS(Z) - 2  \,.
\]
\end{theorem}

\begin{proof}    Proposition \ref{prop-positiveintPic1}
 implies that the hypotheses of Theorem \ref{thm:InjMWsmooth} are satisfied.\end{proof}
  
We then find a bound on $\operatorname{rk} \MW(X/B)$:  
 
\begin{theorem}\label{thm:BoundsFromNoetherMov}  
 
Assume there exists a curve $C \subset B^{sm}$ which is smooth, movable with $C \cdot \Lambda > 0$  and such that   $Z = \pi^{-1}(C)$ is  either  smooth or (birationally) a K3 surface  or a rational elliptic surface\color{black}. 
Then
	\[
	\operatorname{rk} \MW(X/B) \leq 10 (C  \cdot \Lambda) +2 g(C) -2.
	\]
	
\end{theorem}
\begin{proof}
This follows from Theorem \ref{thm:MWNSMov} %and \ref{thm:MWNSPic1}
together with Proposition \ref{prop:Noether} as well as Proposition \ref{prop:K3appearance1}.
\end{proof}

\begin{remark}
For $X$ Calabi-Yau (in particular $K_X = 0$ and hence $\Lambda = -K_B \neq 0$) and $C$ rational, the bound in Theorem (\ref{thm:BoundsFromNoetherMov}) agrees with the physics bound 
\begin{equation}\label{boundX-2}
	\operatorname{rk} \MW(X/B) \leq 10 ( -K_B \cdot C)  -2
	\end{equation}
    argued for in \cite{Martucci:2022krl} for $\dim X =3,4$ by sharpening and generalising to Calabi-Yau 4-folds the conservative bounds of \cite{Lee:2019skh}.
 See also the discussion around (\ref{bound-MW}) in the Introduction.
\end{remark}

\begin{theorem}\label{thm:BoundsFromNoetherPic1}  

Let  $C \subset B^{sm}$  be a smooth 
curve such that either  $\pi^{-1}(C)=Z$ is a smooth surface or $Z$ is (birationally) a K3 surface  or a rational elliptic surface\color{black}.
 Assume also that $C \cdot \Lambda  > 0$ and $\operatorname{rk} \Pic (B)=1$.   Then
	\[
	\operatorname{rk} \MW(X/B) \leq 10 (C  \cdot \Lambda) +2 g(C) -2.
	\]
In particular if $Z$ is birationally a K3 surface 	
\[
	\operatorname{rk} \MW(X/B) \leq 18,\]
    if $Z$ is birationally a rational elliptic surface
   \[
	\operatorname{rk} \MW(X/B) \leq 8.\] 
    \color{black}
\end{theorem}

\begin{proof}
This follows from Theorem %\ref{thm:MWNSMov} %
\ref{thm:MWNSPic1}
together with Proposition \ref{prop:Noether} as well as Proposition \ref{prop:K3appearance1}.
\end{proof}
\section{Reduction to elliptic surfaces}\label{sec:inclusion}

By combining the results of Section \ref{sec:arithmproof} with Proposition \ref{prop:MWtorinj} and the results of Section \ref{sec:geomproof} we obtain inclusions of the Mordell-Weil groups  of sections of $X \to B$ to the Mordell-Weil group of a surface.

Given an algebraically closed sub-field $k$ of $\mathbb{C}$ and a variety $T$ over $k$, the $k$-embeddings of $L=k(T)$ in $\mathbb{C}$ give a subset $V$ of
the set $T(\mathbb{C})$ of (complex) points of $T$. A point lies in the complement of $V$ if and only if it is a point of a proper $k$-subvariety of $T$. This set $V$ is called the set of \emph{very general} $\mathbb{C}$-points of $T$.
\begin{theorem}\label{thm:fieldExt}
    Given a complex elliptic fibration $\pi:X\to B$ with a section, and
    generic fiber $E_K$, where $K={\mathbb C}(B)$, and a covering
    family of movable curves $\{C_t\}_{t\in T}$.  Assume $MW(X/B)=E_K(K)$ is finitely generated.

    For a \emph{very general} $t\in T$, let $h:Z_t=\pi^{-1}(C_t)\to C_t$ be the restricted fibration with generic fiber $E_F$, where $F=\mathbb{C}(C_t)$.

    The natural restriction map $MW(X/B)=E_K(K)\to
    MW(Z_t/C_t)=E_F(F)$  is \emph{injective}.
\end{theorem}

\begin{proof}
   The varieties $X$, $B$, $T$ and the morphisms $\pi$ and $\phi$ are
defined over a countable algebraically closed subfield $k\subset \mathbb{C}$.
The proof of Theorem \ref{main2} can be carried out over $k$ instead of
$\mathbb{C}$. 
\end{proof}

\begin{remark}
$MW (X/B) = E_K (K)$ is finitely generated, unless $E_K\cong
E_0\times_{\mathbb C}{\rm Spec K}$ is a base change of a complex
elliptic curve, i.e., $X$ is birational to $B\times E_0$ for a
complex elliptic curve $E_0$. This is a special case of the
Lang-N\'eron theorem; see for example  \cite{Conrad2006, Lang1959, Silverman1994} and \cite[Th\'eor\`eme 2, with a detailed proof in Appendix 2]{Kahn-Hindry}.\footnote{We  thank Marc Hindry for correspondence on this point.}
     
 \end{remark}
\color{black}

We also have:

\begin{theorem}\label{thm:MWNSMovInjTotal}
 Let $C \subset B^{sm}  %\in \overline{\mathrm{NM}(B)}
 $ be  a   general smooth  movable curve    
 with $\Lambda \cdot C > 0$ and let $\pi^{-1}(C)=Z$.
 Then
\[ 
\MW(X/B)  \subseteq  \MW(Z/C)   \,.
\]
\end{theorem}

\begin{proof}  Proposition \ref{prop-positiveintMov} and 
    Theorem \ref{thm:InjMW} prove that the natural restriction of sections is indeed  an injection  on the free part of the Mordell-Weil groups: $
\MW(X/B)/{\MW(X/B)}_{\rm tors}  \subseteq  \MW(Z/C)/ {\MW(Z/C)}_{\rm tors}  \,$.
    If $C$ is general and movable, we can assume that it is not contained in the divisor where the  torsion sections coincide. Proposition \ref{prop:MWtorinj} gives an injection on the torsion parts.
\end{proof}

\begin{corollary}\label{cor:MWNSSingInjTotal}  Let  $C \subset B^{sm} %\in \overline{\mathrm{NM}(B)}
$  be a  smooth  
curve general enough to be contained neither in the  divisor where the  torsion sections coincide, nor in the discriminant $ \Sigma_{X/B} $.  Assume that $\Lambda = -K_B$ and that $B$ is Fano with $\rho(B)=1$. Let   $\pi^{-1}(C)=Z$.   Then
    \[ 
\MW(X/B)  \subseteq  \MW(Z/C)   \,.
\]
\end{corollary}
\begin{remark}\label{remark:inclusionMWTYpesOfC}
\begin{enumerate}
    \item In the above statements, $Z$ does not need to be smooth.
    \item   The curves  $C$ as in Corollary \ref{cor:MWNSSingInjTotal}  will provide the bounds of the Theorems for fourfolds (Section \ref{sec:4folds}).
\end{enumerate}
\end{remark}

\color{black}
\section{
Bounds for the Mordell-Weil rank of elliptic %Calabi-Yau 
threefolds with $\kappa (X)=0$.}\label{sec:dim3}

When $\kappa (X)=0$, then  $ \Lambda \equiv -K_B$. To obtain strong bounds on ${\rm rk}\,  \MW(X/B)$  via the results of  Sections  \ref{sec:arithmproof} and \ref{sec:geomproof}  it is necessary to minimize  $-K_B \cdot C + 2g(C)-2$.  The existence of rational curves  $C \subset B$ of minimal degree with respect to the anticanonical $K_B$ is a subtle problem in general.
 
We begin by  proving the main result motivating this work, see also  Remark \ref{remark:CY}.

\subsection{Explicit bounds for the Mordell-Weil rank of elliptic Calabi-Yau 
threefolds}\label{subsec:CY3}~

\medskip

 \begin{theorem} \label{boundCY3} Let  $X \to B$  be an elliptic Calabi-Yau 3-fold  
as in Theorem \ref{thm:birminmodel}, with $\Lambda \neq 0$. 
 Then the Mordell-Weil group of $X$ satisfies:\\
(i) ${\rm rk}\, \MW(X/B)\leq 18$ if $B \neq {\mathbb P}^2$;\\
(ii) ${\rm rk}\, \MW(X/B)\leq 28$ if $B  = {\mathbb P}^2$.

 \end{theorem}

\begin{proof} We present two different proofs, the first one, (a),  follows the arithmetic proof of Section \ref{sec:arithmproof}, the second, (b), the geometric proofs of Section  \ref{sec:geomproof}.
\begin{enumerate}[(a)] 
   \item   $X$ has isolated (terminal) singularities. \color{black}  In case i),
{by Theorem \ref{thm:birminmodel} $B$ is the blow up of a Hirzebruch surface and hence it has  a rational fibration.
   There is therefore} a morphism $\theta:B\to {\mathbb P}^1_{\mathbb C}$ whose general fiber is a smooth rational curve with the property that,
for some Zariski open $T\subset {\mathbb P}^1_{\mathbb C}$,
we have that $\theta^{-1}(T)\cong T\times {\mathbb P}^1_{\mathbb C}$.
We may then take the family of rational curves $C_t$ as in Theorem~\ref{main2}
to be the map $T\times{\mathbb P}^1_{\mathbb C}$
obtained from the fibres of $\theta$. Here $d=0$.

In case ii), take $T$ to be an affine open set
in the dual projective plane of lines in $B={\mathbb P}^2_{\mathbb C}$,
where we get a family of rational curves $C_t$ as in Theorem~\ref{main2} with the degree $d=1$.
     \item 
{In case i), 
 the general rational fiber $C$ of the (blow up of the) Hirzebruch surface $B$} is movable and since $C\cdot \Lambda = C \cdot (-K_B) =2$, $Z = \pi^{-1}(C)$ is (birationally) a K3 surface. Then by Theorem \ref{thm:BoundsFromNoetherMov}, ${\rm rk}\, \MW(X/B)\leq 18$. 

{In case ii)} we can take $C$ to be 
a general hyperplane class. Hence $C$ is movable and $Z = \pi^{-1}(C)$ is smooth and 
 $C \cdot (-K_B) =3$. Theorem \ref{thm:BoundsFromNoetherMov}  implies ${\rm rk }\, \MW(X/B)\leq 28$. 
\end{enumerate} \end{proof}

 \begin{remark}\label{remark:CY}
For $X$ a smooth  Calabi-Yau threefold and $B$ a ruled surface, the  bound 
$\operatorname{rk} \MW(X/B) \leq 18$
appeared in the physics literature  together  with a list of  possible torsion groups \cite{Lee:2022swr} \cite{HajoujiOehlmann}. See Appendix \ref{sec:Appendix} for details.
     \color{black}
 \end{remark}

 When $X$ is a Calabi-Yau threefold, or more generally a threefold with $K_X \simeq \mathcal{O}_X$ or $h^1(\mathcal{O}_X)=1$,  by  Theorem \ref{thm:birminmodel},  $B$ is either a toric smooth surface or the blowup of a toric smooth surface.    Then there  exists a minimal generator  of an extremal ray which represents a base point free movable  rational curve, as  toric nef divisors are base point free \cite[Theorem 1.6]{Mavlyutov2000}; then the proof of Theorem \ref{them:boundKodZeroBNonLDP}, along with Corollary \ref{cor:boundKodZeroBSmooth} leads to the bounds described in the next subsection.

\subsection{Explicit bounds for the Mordell-Weil rank of all elliptic threefolds with $\kappa(X)=0$.  %either $\omega_X \simeq \mathcal{O}_X$ or $h^1(X, \mathcal{O}_X) \geq 1$
}\label{subsec:k0}~

\medskip

\begin{theorem}\label{them:boundKodZeroBNonLDP}
    Let  $X \to B$  be an elliptic  threefold as in Theorem \ref{thm:birminmodel}, 
   with  $\Lambda \neq 0$,  $\kappa(X)=0$. 
 Then:\\
(i) If  $\operatorname{rk} {\rm Pic}(B) >1$, then ${\rm rk}\, \MW(X/B)\leq 18$;\\
(ii) if $\operatorname{rk} {\rm Pic}(B) =1$ and $B = {\mathbb P}^2$, then  ${\rm rk }\, \MW(X/B)\leq 28$.
\end{theorem}
\begin{proof}   If  $\operatorname{rk} {\rm Pic}(B) =1$,  $B$ is a log del Pezzo surface, if it is smooth, then  $B  \simeq {\mathbb P}^2$.  If  $\operatorname{rk} \Pic(B) >1$, $B$ is birational to a Mori fiber space with smooth general fiber as klt surface  singularities are isolated quotient singularities. The  statements follow as in Theorem \ref{boundCY3}. \end{proof}

 \begin{corollary} \label{cor:boundKodZeroBSmooth} Let  $X \to B$  be an elliptic  threefold, %X$ smooth and  $B$ smooth 
as in Theorem \ref{thm:birminmodel},   with  $\Lambda \neq 0$,  $\kappa(X)=0$ and either $\omega_X \simeq \mathcal{O}_X$  or $h^1(X, \mathcal{O}_X) \geq 1$. 
 Then the Mordell-Weil group of $X$ satisfies:\\
(i) ${\rm rk}\, \MW(X/B)\leq 18$ if $B \neq {\mathbb P}^2$;\\
(ii) ${\rm rk }\, \MW(X/B)\leq 28$ if $B = {\mathbb P}^2$.
  \end{corollary}

\begin{proof} For $X$ a threefold as in Theorem \ref{thm:birminmodel} (6), $B$ is smooth. \end{proof}
\color{black}

\begin{example}\label{ex:k0}
Let $X=(E \times T)/G$, where $E$ is an elliptic curve, $T$ an elliptic K3 surface with section and $G$ a finite group acting faithfully on $E$ as a subgroup of translations and the action on $T$ leaves the elliptic fibration with section and whole $H^{2,0}$ space invariant. Then $\omega_X \simeq \mathcal{O}_X$, $h^1(X, \mathcal{O}_X) \geq 1$ and $\Lambda \neq 0$.
 Such $G$ and $T$ exist, see for example \cite{GarbagnatiSarti2009}, and   $ {\rm rk}\, \MW(X/B)\leq 18$.

This type of threefolds (without the assumption of $T$ being elliptically fibered) are, together with Calabi-Yau,  one of the 9 classes of threefolds with $\kappa (X)= 0$ \cite{KawamataK0, Morrison1986}.  

\end{example}

\begin{remark}\label{remark:producttype}   The threefolds  $X$ as in Example \ref{ex:k0} do not satisfy the hypothesis of \cite{DiCerboSvaldi, Filipazzi:2021dcw} and in general of \cite{Birkar_2024},  which  imply (birational) boundedness of  the family, and hence boundedness of ${\rm rk}{\MW}(X/B)$.
 In fact, $X$ is not a Calabi-Yau as defined in \cite{Filipazzi:2021dcw},  $X$ is  instead a threefold of product type \cite{DiCerboSvaldi}, possibly singular and  $B$ is not rationally connected \cite{Birkar_2024}.

 These threefolds illustrate the properties described in the  Remark following Theorem 1.4 of \cite{Birkar_2024}. In fact, the Mori fibration  
$B \to E/G$ is the maximal rationally connected (MRC, as in Definition \ref{def:MRC}) fibration. The induced fibration 
$X \to E/G$ is isotrivial.
The non-rationally-connected part of the base contributes only isotrivially to the geometry, in agreement with \cite{Birkar_2024}. 
At the same time, the presence of the elliptic  curve $E$ leads to families which are not birationally bounded, even though the Mordell–Weil rank depends only on the fibration of the isotrivial surface and remains uniformly bounded by Theorem \ref{thm:boundKodZero3}.
\end{remark}

\medskip

  Let $X \to B$  be as in Theorem \ref{thm:birminmodel}, with $\kappa(X)=0$ and  %$X \to B$, elliptic with section and 
 $B$  a  singular log del Pezzo surface (that is  $B$ is a  singular $\mathbb{Q}$-factorial surface with at worst klt singularities,  $-K_B$ ample)  with $\operatorname{rk}{\Pic}(B)=1$. This is  the case excluded  by the hypothesis of  Theorem \ref{them:boundKodZeroBNonLDP}.

 Such singular  log del Pezzo surfaces %with  \( \rho(B)=1 \)  and klt singularities 
 are classified in \cite{YeGorLogDelPezzo}  and \cite{Lacini2024}, following \cite{KeelMcKernan1999} (see also \cite{Nagaoka2025reviewlacinisclassificationrank}).
 The classification provides a list of  $24$ non-Gorenstein families,  %labelled LDP$1$- -LDP$24$ in \cite{Lacini2024},  
 $29$ Gorenstein surfaces, defined  up to isomorphisms, and an infinite family of  Gorenstein surfaces %with  two  $D_4$  singularities 
 \cite{YeGorLogDelPezzo} (see also \cite{Alexeev1994} for boundedness results).

General results of \cite[Thm. 6.5]{XuSRC2012}  and \cite[Lemma 13.7]{KeelMcKernan1999} prove that $B^{sm}$ is strongly rationally connected, that is,  there  exists a family of very free rational curves $\{C_t\}  \subset  B^{sm}$.
However, these curves  cannot be used to obtain bounds on ${\rm rk}\, \MW(X/B)$, as  the following holds:
\begin{lemma}[Lemma 13.7, \cite{KeelMcKernan1999}]\label{lemma:KeelMcK} Let $B$ be as above.
    Suppose $\mathbb{P}^1 \simeq C \subset B^{sm}$% is a smooth rational curve
    . Then either $B=\mathbb{P}^2$ or $B=\mathbb{P}[1,1, n]$ is the cone over a rational normal curve of degree $n$ with zero section $\sigma_0$
    and $C  \in | 
    \sigma_0|$.
\end{lemma}
If  $B=\mathbb{P}[1,1, n]$ and $C$ are  as in Lemma \ref{lemma:KeelMcK} the numerical invariants 
${(g(C), -K_B \cdot C)}$ grow with $n$ and  $-10K_B \cdot C + 2g(C)-2$  cannot provide a uniform bound for $\operatorname{rk}\MW(X/B)$.

\smallskip

Nevertheless, in    Theorem \ref{thm:boundKodZero3} we will   prove the existence of an explicit uniform bound.
 Unlike the Calabi–Yau case, where the universal bound is $28$, the broader class of elliptic threefolds with $\kappa (X)=0$ requires a larger uniform bound.

We need the following:

\begin{lemma}[The Bogomolov Bound, Lemma 1.8.1, \cite{KeelMcKernan1999}]\label{lemma:bog} Let $B$ be a  singular log del Pezzo surface  with $\operatorname{rk}{\Pic}(B)=1$. Then $B$ has at most $4$ singular points.
\end{lemma}
\begin{lemma}[Main Theorem and Corollary 4.7, \cite{GrassiSingBase1993}]\label{lemma:SingsBaseK0} Let $X \to B$, as in Theorem \ref{thm:birminmodel}, with $\kappa(X)=0$ and  
 $B$  a  singular log del Pezzo surface with $\operatorname{rk}{\Pic}(B)=1$. 
Let 
$f:\widetilde B\rightarrow B$ be 
the minimal resolution of $B$. Then
\begin{enumerate}[(a)]
    \item  $f^{-1}(P)$ is a Hirzebruch--Jung string of type
$<n_P,q_P>$,  $\forall P \in Sing(B)$,
\item 
$n_P\le 6$, $\forall P \in Sing(B)$,
\item
$2 \chi(\mathcal O_X)=
\sum (1-1/n_P)$.
\end{enumerate}   
\end{lemma}

Lemma \ref{lemma:SingsBaseK0} implies that the  infinite series of singular Del Pezzo surfaces $\mathbb{P}[1,1,n]$ of Lemma \ref{lemma:KeelMcK}
 do not occur.
 
\begin{corollary}\label{cor:numbasketsK0} Let $X \to B$, as in Theorem \ref{thm:birminmodel}, with $\kappa(X)=0$ and  
 $B$  a  singular log del Pezzo surface with $\operatorname{rk}{\Pic}(B)=1$.  Then the possible sets of integers $n_p$ which can occur in the
basket of singularities of $B$ are
$
(2,3,6),
(2,4,4),
(3,3,3).$

Moreover $\chi(\mathcal O_X)=1$.
    
\end{corollary}
\begin{proof}
Since  $0 \leq \chi(\mathcal O_X) \leq 4$ [Theorem 3.1, \cite{KawamataK0}],   the statement follows from Lemma \ref{lemma:bog} and Lemma \ref{lemma:SingsBaseK0}. Note that the configuration $(2,2,2,2)$  is excluded by classical results  originating with the work of Du Val (see for example \cite{YeGorLogDelPezzo}). \end{proof}
\begin{corollary}\label{cor:basketK0} Let $B$ as in Corollary \ref{cor:numbasketsK0}. 
\begin{enumerate}[(i)]
    \item If $B$ is Gorenstein, the possible singularities of $B$ are:\\
$(A_5+A_2+A_1), \
(A_1 + 2 A_3), \ 3A_2$.\\
Moreover, each basket of singularities determines uniquely
a weak del Pezzo surface $B$.
\item If $B$ is not Gorenstein, the possible baskets of singularities are:\\
$(A_1, (-3), (-6))$, $(A_1, A_2, (-6))$, $(A_1, A_5, (-3))$, $(A_1, (-4), (-4))$, \\$(A_1, A_3, (-4))$, $ ((-3),(-3), (-3))$, $(A_2, A_2, (-3))$, $(A_2,  (-3), (-3))$. \\
Moreover, there exists
a finite number of
families with the allowed baskets of singularities.
\end{enumerate}
The singularities are cyclic quotients.
\end{corollary}

\begin{proof} By Corollary \ref{cor:numbasketsK0} and \cite{YeGorLogDelPezzo}, \cite{KeelMcKernan1999} and \cite{Lacini2024}.  \end{proof}

\begin{lemma}\cite{Demazure1980_PezzoIV}, \cite{FujinoBPFree}\label{lemma:bpfIndex}~

Let $B$ as in Corollary \ref{cor:basketK0}. There exists an integer $m$ such that  $|-m K_B|$  is base point free. In particular:
\begin{enumerate}[(i)]
    \item  If $B$ is Gorenstein and it has basket $(A_5+A_2+A_1)$, then $|-2K_B|$ is base point free, in the other cases $|-K_B|$ is base point free.
\item If $B$ is not Gorenstein  then $|-2 i(B) \ K_B|$ is base point free, where $i(B)$ is the Cartier index of $B$.
\end{enumerate}
    
\end{lemma}

\begin{proof} The existence of  an integer  $m$  such that  $|-m K_B|$  is base point free  follows from classical results of the Minimal Model Program (Kawamata-Shokurov). 
  If $B$ is Gorenstein Noether's formula gives $K_B^2= 9 -e $, where $e$ is the number of curves on the minimal resolution of $B$, so in case (i) $K_B^2=1$ and  $K_B^2 \geq 2$ otherwise. 
(i)  then follows from  \cite[Theorem 1, IV]{Demazure1980_PezzoIV} and (ii) from \cite[Corollary 1.5]{FujinoBPFree}. \end{proof}
Properties of log del Pezzo surfaces with $i(B)=2$  have been studied and classified in \cite{AlexeevNikulin}, the ones with $i(B)=3$  in \cite{FujitaYasutake}  and \cite{CortiHeuberger}, also in the context of mirror symmetry and mutations.

We are now ready to state the general result:

\begin{theorem}\label{thm:boundKodZero3}
Let \( X \to B \) be an elliptic threefold as in Theorem~\ref{thm:birminmodel}, with \( \kappa(X)=0 \) and \( \Lambda \neq 0 \). 
Then \[\operatorname{rk}\,\MW(X/B) \leq 810,\]
 the bound is independent of $X$, $B$, and the torsion index of $K_X$.
In addition one of the following holds:
\begin{enumerate}
    \item If \( \rho(B)>1 \), then
   $ \operatorname{rk}\,\MW(X/B) \leq 18.$
%\item If \( \rho(B)=1 \) and  \( B \simeq \mathbb{P}^2 \), then
  % $ \operatorname{rank}\,\MW(X/B) \leq 28.$
    \item If $\rho(B)=1 $ and $B$ is smooth, then  $\operatorname{rk}\,\MW(X/B) \leq 28$, and  $B \simeq \mathbb{P}^2$. 
    \item If   \( \rho(B)=1 \) and $B$ is singular,  the bounds on $\operatorname{rk}\,\MW(X/B) $ are listed in Tables \ref{tab:gorenstein-delpezzo}, \ref{tab:gorenstein-delpezzo-2K} and \ref{tab:delpezzo}.
\end{enumerate} 

\end{theorem}

\begin{proof} 

(1) and (2) are proved in Theorem \ref{them:boundKodZeroBNonLDP}.   Let now $B$ be as in (3),  $|-mK_B|$  as in Lemma \ref{lemma:bpfIndex},  and $C\in |-mK_B|$ be a smooth general curve.   $C \subset B^{sm}$ satisfies the hypothesis of Theorem \ref{thm:BoundsFromNoetherMov}, by Bertini's Theorem, since the singularities of $B$ are finite.
 Then again by Theorem \ref{thm:BoundsFromNoetherMov} $ \operatorname{rk}\,\MW(X/B) \leq 10 C  \cdot (- K_B) +2 g(C) -2$,  and  this bound  depends only on the  basket of singularities,   by Corollary \ref{cor:basketK0} and Lemma \ref{lemma:bpfIndex}. 

%\newpage

\begin{table}[ht]
\caption{\scriptsize  Numerical invariants for a general smooth curve $C\in |-K_B|$, $B$ Gorenstein.}
\label{tab:gorenstein-delpezzo}
\centering
\small
\renewcommand{\arraystretch}{1.3}
\begin{tabular}{|l|c|c|c|c|}
\hline
\textbf{Basket} & \textbf{$K_B^2$} & \textbf{$-K_B\cdot C$} & \textbf{$g(C)$} & \textbf{$-10K_B\cdot C+2g(C)-2$} \\
\hline
$(A_1+2A_3)$ & $2$ & $2$ & $1$ & $20$ \\
\hline
$(3A_2)$ & $3$ & $3$ & $1$ & $30$ \\
\hline
\end{tabular}
\end{table}

\begin{table}[ht]
\caption{\scriptsize  Numerical invariants for a general smooth curve $C\in |-2K_B|$, $B$ Gorenstein.}
\label{tab:gorenstein-delpezzo-2K}
\centering
\small
\renewcommand{\arraystretch}{1.5}
\begin{tabular}{|l|c|c|c|c|}
\hline
\textbf{Basket} & \textbf{$K_B^2$} & \textbf{$-K_B\cdot C$} & \textbf{$g(C)$} & \textbf{$-10K_B\cdot C+2g(C)-2$} \\
\hline
$(A_1+A_2+A_5)$ & $1$ & $2$ & $2$ & $22$ \\
\hline
\end{tabular}
\end{table}
\begin{table}[ht]
\caption{ \scriptsize Numerical invariants associated with a general smooth curve $C\in |-2i(B)K_B|$, $B$ non Gorenstein.}
\label{tab:delpezzo}
%\vspace{0.1mm}
\centering
\small
\renewcommand{\arraystretch}{1.5}
\begin{tabular}{|l|c|c|c|c|c|}
\hline
\textbf{Basket} & $i(B)$ & $K_B^2$ &  $-K_B \cdot C %=2i(B)K_B^2
$ & $g(C)$ & $-10K_B\cdot C+2g(C)-2$ \\
\hline $(A_1,-3,-6)$ & $3$ & $9$ & $54$ & $136$ & $810$ \\
\hline
$(A_1,A_2,-6)$ & $3$ & $\frac{23}{3}$ & $46$ & $116$ & $690$ \\
\hline
$(A_1,-3,A_5)$ & $3$ & $\frac{7}{3}$ & $14$ & $36$ & $210$ \\
\hline
$(A_1,-4,-4)$ & $2$ & $8$ & $32$ & $49$ & $416$ \\
\hline
$(A_1,A_3,-4)$ & $2$ & $5$ & $20$ & $31$ & $260$ \\
\hline
$(-3,-3,-3)$ & $3$ & $7$ & $42$ & $106$ & $630$ \\
\hline
$(A_2,A_2,-3)$ & $3$ & $\frac{13}{3}$ & $26$ & $66$ & $390$ \\
\hline
$(A_2,-3,-3)$ & $3$ & $\frac{17}{3}$ & $34$ & $86$ & $510$ \\
\hline
\end{tabular}

\end{table}
\end{proof}

\medskip
\medskip

\newpage

\subsection{$\MW(X/B)_{\rm tors}$ and further bounds on $\rm rk \MW(X/B)$}\label{subsec:tors}~

\smallskip
The $19$ possible torsion subgroups 
of an elliptic fibration over a rational curve
are listed in Table \ref{table:MirPer},
following the results in \cite{MirandaPersson1989}.
\begin{table}[h]
\caption{}\label{table:MirPer}
\centering
\resizebox{\textwidth}{!}{%
\begin{tabular}{|c|c|}
\hline
  $s$ as in $(1,s)$ (cyclic) & $(r,s),\ r>1$ (product) \\
\hline
$0,2,3,4,5,6,7,8,9,10,12$
& $(2,2),(2,4),{(2,6)}, (2,8),(3,3),(3,6),(4,4),(5,5)$ \\
\hline
\end{tabular}%
}
\end{table}

These bound the possibilities for the elliptic $3$-folds we are interested in as well:

 \begin{theorem}\label{thm:TorsionCY3}
 Let  $X \to B$  be an elliptic $3$-fold, 
 as in Theorem \ref{thm:birminmodel},
 with $\kappa(X)=0$
 and $\Lambda \neq 0$.
 
 \begin{enumerate}
\item If either  \( \rho(B)>1 \) or $B = {\mathbb P}^2$, 
the possible subgroups of $\MW(X/B)_{\mathrm{tors}}$ 
are the torsion subgroups  $ G_f \cong \mathbb{Z}/r \times \mathbb{Z}/s$ listed in Table \ref{table:BsmoothTor}.

\item  If  \( \rho(B)>1 \)  and $\MW(X/B)_{\mathrm{tors}} \neq 0$, $\operatorname{rk} \MW(X/B)$ satisfies the additional bounds
listed in Table \ref{table:NonZeroTor1}.
\end{enumerate}
 \end{theorem}
 
\begin{table}[h]
\caption{}\label{table:BsmoothTor}
\centering
\resizebox{\textwidth}{!}{%
\centering
\begin{tabular}{|c|c|c|}
\hline
 & $s$ as in $(1,s)$ (cyclic) & $(r,s),\ r>1$ (product) \\
\hline
$B,\ \rho(B)>1$
& $0,2,3,4,5,6,7,8$
& $(2,2),(2,4),(2,6),(3,3),(4,4)$ \\
\hline
$B=\mathbb{P}^2$
& $0,2,3,4,5,6,9,10$
& $(2,2),(2,4),(3,3),(3,6)$ \\
\hline
\end{tabular}%
}
\end{table}

\begin{table}[h]
\caption{}\label{table:NonZeroTor1}
\centering
\begin{tabular}{|c|c|c|c|c|c|c|c||c|c|c|c|c|}
\hline
 & \multicolumn{7}{c||}{$s$ as in $(1,s)$ (cyclic) } & \multicolumn{5}{c|}{$(r,s),\ r>1$ (product) } \\
\hline
$G_f$ 
& 2 & 3 & 4 & 5 & 6 & 7 & 8 
& (2,2) & (2,4) & (2,6) & (3,3) & (4,4) \\
\hline
$\max \operatorname{rk}{\rm MW}$ 
& 10 & 6 & 4 & 2 & 4 & 0 & 0 
& 6 & 4 & 0 & 2 & 0 \\
\hline
\end{tabular}
\end{table}

\begin{proof}
\begin{enumerate}
    \item For any such $B$ there exists a smooth  movable rational curve $C \subset B^{sm}$,   \cite{XuSRC2012}. We can choose a general $C$ such that $Z= \pi ^{-1}(C)$  is a smooth (minimal)  surface and 
$\MW(X/B)_{\rm tors}  \subseteq \MW(Z/C)_{\rm tors}  $, by Proposition  \ref{prop:MWtorinj}.  The table is  then derived  from \cite[Lemma1.1, Table 4.5]{MirandaPersson1989}; note that $(r,s)$ is determined by $K_B \cdot C$.
\item In \cite{Shimada2000}, the possible values of ${\rm rk\,MW}(X/B)$ are determined for the case in which \(Z=\pi^{-1}(C)\) is a smooth K3 surface, which is the case whenever \(\rho(B)>1\).
\end{enumerate}\end{proof}
\medskip

\subsection{$\kappa(X)<0$}~
\medskip

 Similarly, we can  obtain explicit bounds on the Mordell-Weil rank of elliptic threefolds  $X$ with 
$\kappa(X) <0 $:

\begin{proposition}\label{prop:KodNeg}
Let \( X \to B \) be an elliptic threefold as in Theorem~\ref{thm:birminmodel}, with \( \kappa(X)<0 \) and \( \Lambda \neq 0 \). Then:
\begin{enumerate}
    \item If \( \rho(B)>1 \),
   $ 
    \operatorname{rk}\,\MW(X/B) \leq 8.$
  \item  If  \( \rho(B)=1 \)  and $B$ is smooth,  $\operatorname{rk}\,\MW(X/B) \leq 18.$
\end{enumerate} 
\end{proposition}
\begin{proof}  Since $\kappa(X)<0$, $K_B+ \Lambda <0$. In case (1) the general  curve $C$  of the Mori fibration is rational with $\Lambda \cdot C=1$, and $Z \to C$ is a rational elliptic surface. In case (2), $B \simeq \mathbb{P}^2$,  either $\Lambda \cdot C=1$ or $2$ and $Z \to C$ is  either a rational  or K3 elliptic surface. \end{proof}

\begin{example}\label{ex:K<0} %$X$ is the resolution of a Weierstrass model on $B$, where $B$ is a ruled surface and $\mathcal{L}^{\otimes r} \simeq {\omega^{-1}_B}, \ r >1 $.
Relevant  examples  of this type of threefolds are  in the papers \cite{Totaro2008} and \cite{PrendergastSmith2012}.
\end{example}

\section{ Bounds for the Mordell-Weil rank of elliptic 
fourfolds with  
$\kappa(X) \leq 0$} \label{sec:4folds}

 In this section, we prove bounds for elliptic fourfolds with $K_X \equiv 0$.

We work under the following hypothesis:

 \color{black}
\begin{assumption}\label{Assumption} Let $X \to B$ as in Theorem \ref{thm:birminmodel}, 
with $K_X  \equiv 0$ such that $\Lambda = - K_B \neq 0$. Let $B \dasharrow B' \to  U $ as in  Definition \ref{def:MoriFS}.
\begin{enumerate}[(a)]
\item   If  $\dim U= 1$, we assume that $B'$ has terminal singularities.\color{black}
\item  If  $\dim U= 0$, that is $\rho(B')=1$, we assume that 
$B'$ is smooth and furthermore
 that it is general in its deformation family if in addition $r(B')=1$, where 
 $r(B')$ is the Fano index of $K_B$, i.e. $-K_{B'} = r({B'}) H$ with $H$ the generator of $\Pic(B')$.

\end{enumerate}

 \end{assumption}

\begin{remark}\label{remark:OnAssumption4Folds} 
\begin{enumerate}[(a)]
\item The assumption that $B'$  has  terminal singularities is a natural generalization of the  lower dimensional case, where $B$  is smooth.  
\item If $B'$ is Fano with Gorenstein terminal singularities, then $B'$ is smoothable \cite{Namikawa1997}, so  if $B'$ is Gorenstein and general, the assumption is satisfied. 

\item There are some  results on the existence of free curves on Fano klt varieties \cite{jovinellyLehemannRiedl-DeJong2026freecurvesfundamentalgroups}.
\item There exists a partial classification of  factorial Fano threefolds $B'$ with 
   terminal singularities in \cite{KuznetsovProKhorov2023}, see also \cite{bayer2025mukai} and \cite{kuznetsov2023onenodal}.  Recall that $\mathbb{Q}$-factorial Gorenstein is factorial.
  \end{enumerate}
  \end{remark}

 We prove the following:
 \begin{theorem}\label{thm-4foldgen} 
 Under the hypothesis as in Assumption \ref{Assumption},
      \[\operatorname{rk} \MW(X/B)  \leq  38. \]
      In addition, if $B' \to U$,  is a  Mori fiber space with  $\dim U >0$, 
      \begin{enumerate}[(i)]
          \item $\operatorname{rk} \MW(X/B)  \leq  28 $  when the general fiber  of the Mori fibration is $\mathbb{P}^2$,
          \item   and otherwise,
$\operatorname{rk} \MW(X/B)  \leq  18. $
      \end{enumerate}

 \end{theorem}

The proof proceeds by considering separately the three possible dimensions of  $U$ of the Mori fiber space $B' \to U$. In each case we construct a curve $C$ satisfying the hypotheses of Theorem \ref{thm:BoundsFromNoetherPic1}.

\begin{corollary}\label{cor:CY4bounds}
In particular, the bounds of Theorem \ref{thm-4foldgen}  apply to elliptic Calabi-Yau fourfolds with the additional Assumption \ref{Assumption}.
\end{corollary}

\begin{remark}\label{remark:BDS4}
In addition to Calabi-Yau fourfolds in the sense of Definition \ref{def-CY}, the bounds of Theorem \ref{thm-4foldgen} apply also to other fourfolds, for example fourfolds  of the form $(E \times T)/G$, where one of the following holds (such examples exist, as in Example  \ref{ex:k0}):
\begin{enumerate}
\item 
$E$ is an elliptic curve, $T$ an elliptic Calabi-Yau threefold  with section and $G$ a finite group acting faithfully on $E$ as a subgroup of translations and the action on $T$ preserves the elliptic fibration  with  section and leaves $H^{3,0}$ invariant;
\item 
$T$ is an abelian surface and $E$ is an elliptic K3 surface  with section and $G$ a finite group acting faithfully on $T$ as a subgroup of translations and the action on $E$ leaves invariant the elliptic fibration with  section and  $H^{2,0}$.
\end{enumerate}

\end{remark}

\smallskip

\begin{remark} 
The fourfolds described in Remark  \ref{remark:BDS4} have the same isotrivial behaviour as  Example \ref{ex:k0}  and Remark \ref{remark:producttype}, and in \cite{Birkar_2024}. In case (1) 
$
X = (E \times T)/G,
$
where $T \to B_0$ is an elliptic Calabi–Yau threefold. The base of the induced elliptic fibration 
$
B = (E/G) \times B_0
$
 is uniruled but not rationally connected. The MRC fibration is given by the projection
$B \dashrightarrow E/G,$  whenever $B_0$ is uniruled,
and the induced map
$X \dashrightarrow E/G$
has fibers isomorphic to $T$. Hence this fibration is isotrivial.

A similar description applies to case (2), where the abelian surface factor plays the role of the non-rationally-connected base. 

In both cases, the non-rationally-connected directions contribute only isotrivially, while the variation relevant for the Mordell–Weil group is governed by the geometry of the elliptic fibration on the remaining factor.

This explains why, while these families are not birationally bounded,  the Mordell–Weil rank remains uniformly bounded as in Theorem \ref{thm-4foldgen}.
\end{remark}

The same reasoning, with Assumption \ref{Assumption}, provides bounds on the Mordell-Weil rank of elliptic fourfolds with
$\kappa(X) <0$, as in Example \ref{ex:K<0}.

\vskip 0.2in

We reduce the proof of  Theorem \ref{thm-4foldgen}, to the existence of suitable curves $C \subset B$ satisfying the hypotheses of  either Theorems \ref{main2}, \ref{thm:BoundsFromNoetherMov} or \ref{thm:BoundsFromNoetherPic1}:

\begin{theorem} \label{them-existencecurve} Let $B'$  as in Theorem \ref{thm:MMP}, with $B'$ smooth. Then there exists a smooth curve $C \subset B$ such that 
	$\operatorname{rk} \MW(X/B) \leq 10 (-K_{B'} \cdot C ) +2 g(C) -2.$

\end{theorem}

\begin{proof} 
Theorem \ref{thm:MMP} implies that in cases (1) and (2) there exists a curve $C$ smooth, movable and rational  such that  $Z$ is smooth, and  Theorems \ref{main2}  and \ref{thm:BoundsFromNoetherMov} both apply. 
Otherwise, in case (3), $\operatorname{rk} \Pic (B')=1$ and we show that there exists a smooth curve $C$  such that either $Z=\pi^{-1}(C)$ is smooth  or $Z$ is birationally a K3 surface and  Theorem \ref{thm:BoundsFromNoetherPic1}  applies.

We prove the existence of such a  curve $C$ in the  sections that follow.
\end{proof}

\begin{corollary} \label{coroll-ConB} Let $B$, $B'$ and $\mu: B \to B'$ be  as in Assumption \ref{Assumption}. Then there exists a smooth curve $C$ in $B$ such that 
    \[
	\operatorname{rk} \MW(X/B) \leq 10 (-K_B \cdot C ) +2 g(C) -2.
	\]
\end{corollary}
\begin{proof} By the structure Theorem \ref{thm:MMP}  if $\rho (B) >1$, either cases (1) and (2) apply and the curve $C$ can be taken in the open set where $\mu$ is an isomorphism. If $\rho(B)=1$, then $B= B'$. \end{proof}
We then  prove {Theorem \ref{thm-4foldgen} 
 by combining the optimal bounds obtained via Theorem \ref{them-existencecurve}  for all possible $B$ under the stated assumptions.}

\begin{proposition}\label{prop:4foldsMoriFib1}
 In case (1)  of  Theorem \ref{thm:MMP}
	\[
	\operatorname{rk} \MW(X/B) \leq 18.
	\]
\end{proposition}

\begin{proof}	Since $B'$ has at worst klt singularities, the general fiber of $f: B' \to U $ is in fact a smooth rational movable curve with $Z$ smooth and
	if  $\Lambda \neq 0$,  Theorem \ref{thm:BoundsFromNoetherMov}  and adjunction imply: $\operatorname{rk} \MW(X/B) \leq - 10 (C  \cdot K_B)  -2= 18.$ \end{proof}

\begin{proposition}\label{prop:4foldsMoriFib2} In case (2)  of  Theorem \ref{thm:MMP},
$\operatorname{rk} \MW(X/B) \leq 28$ if the  fiber is $\mathbb{P}^2$ and 
$\operatorname{rk}\MW(X/B) \leq 18$  otherwise.
\end{proposition}

\begin{proof}  By Assumption \ref{Assumption}, the  general fiber  $F$ of $f: B' \to U $ is a smooth del Pezzo surface, as terminal singularities are isolated, and $\pi^{-1}(F) $ is a smooth threefold with    Kodaira dimension zero and conclude as in  Theorem \ref{them:boundKodZeroBNonLDP}. \end{proof}

 \begin{theorem}\label{thm:TorsionCY4} In case (1) and (2)  of  Theorem \ref{thm:MMP}, the torsion groups  ${\MW(X/B)}_{\rm tors}$ are  those listed  in  the tables of Theorem \ref{thm:TorsionCY3}.
 \end{theorem}

\begin{proof}  We argue as in Propositions \ref{prop:4foldsMoriFib1} and \ref{prop:4foldsMoriFib2} and conclude as in Theorem \ref{thm:TorsionCY3}. \end{proof}

\vskip 0.2in

In case (3) of  Theorem \ref{thm:MMP} $B$ is a Fano threefold, which we assume to be smooth, as in Assumption \ref{Assumption}  with $\rho(B)%:=\operatorname{rk} \operatorname{Pic} (B)
=1$.
 As in the proof of Corollary \ref{coroll-ConB}, we can set $B = B'$.

Let $r(B)$ be  the index of $K_B$, i.e. $-K_B = r(B) H$ with $H$ the generator of $\Pic(B)$.
Recall the following, by now classic, result:

\begin{theorem}\cite[Table 6.5]{Iskovskih1978FanoII} Let $B$ be a smooth Fano threefold  with  $\rho  (B)=1$ and index $r(B)$. 
\begin{enumerate}[(i)]
    \item  If $r(B)=3,4$, then $B= \mathbb{P}^3$ or  a quadric $Q \subset \mathbb{P}^4$.
    \item  If  $r(B)=2$, then  there are $5$ deformation types and $B$ is called a del Pezzo threefold.
    \item  If $r(B)=1$, then $H^3=-K_B^3 $ and there are 10 deformation types.\\ Furthermore $4 \leq H^3 \leq 18$ unless:
    \begin{itemize}
    \item  $H^3 =2$ and $B$ is a double cover of $\mathbb{P}^3$ with branch locus a divisor of degree $6$%hypersurface of degree $6$ in $\mathbb{P}(1^4,3)$
    ,
    \item  $H^3 = 22$ and $B$ is the zero locus of $(\wedge^2\mathcal{U}^{v})^{\oplus 3}$ on $Gr(3,7)$.
    \end{itemize}
    
    In addition $-K_B$ is very ample unless
    \begin{enumerate}
  \item $H^3 =2$ and $B$ is a double cover of $\mathbb{P}^3$ with branch locus a divisor of degree $6$,
\item $H^3=4$ and $B$ is the double cover of a quadric in $\mathbb{P}^4$ branched along an octic,
\item $H^3=10$ and $B$ is the double cover of a section of Plücker embedding of $Gr(2,5)$.
    \end{enumerate}
        
\end{enumerate}
    
\end{theorem}

{We now analyze the cases in turn.}\\

\vskip 0.4in
\subsection{Case (i):  
$\rho (B)=1$, $r(B)=3,4$.}~
\vskip 0.1in

\begin{proposition} Let $B$ be a smooth Fano threefold  with $\rho(B)=1$, $r(B)=3,4$. Then 
\[\operatorname{rk} \MW(X/B)  \leq   38.\]
\end{proposition}
\begin{proof}Each  such Fano contains a  movable line $\ell$, that is a smooth rational curve with $H \cdot 
\ell=1$, for $H$ the ample generator of $\operatorname{Pic}(B)$, such that $Z= \pi^{-1}(\ell)$ is smooth. Then $-K_B \cdot \ell= r(B)$, and
    \[\operatorname{rk} \MW(X/B)  \leq  10(-K_B \cdot \ell ) -2\leq 38.\] \end{proof}

\subsection{Case (ii):  
$\operatorname{rk} \Pic (B)=1$, $r(B)=2$.}~
\vskip 0.2in
 In the following Propositions \ref{prop:delpezzogeq32} and \ref{prop:delpezzoleq32} we recall properties of a del Pezzo threefold $B$ (that is $B$  Fano, such that $r(B)=2$ and $\rho(B)=1$) up to deformation type:
\begin{proposition}\label{prop:delpezzogeq32} \cite[Proposition 3.21]{Prokhorov_2025}
   Let $B$ be a smooth del Pezzo threefold %(that is  $r(B)=2$ and $\rho(B)=1$) 
   such that $ (-K_B)^3 \geq 32$. Then 
   \begin{enumerate}
\item there exists a line $\ell \subset B$ through any point in $B$,
\item $F_1(B)$,  the Hilbert scheme parametrizing lines on $B$, is reduced, nonsingular, and each of its components is two-dimensional,
\item a  line $\ell$  corresponding to the general point of a component of $F_1(B)$ has normal bundle $N_{\ell / B}=\mathcal O_\ell \oplus \mathcal O_\ell $.
   \end{enumerate}
\end{proposition}

\begin{proposition}\label{prop:delpezzoleq32}\cite[Table 1]{Prokhorov_2025} Let $B$ be a smooth del Pezzo threefold 
such that $ (-K_B)^3 <32$. Then 
\begin{enumerate}[(i)]
    \item  If $(-K_B)^3 =24$, $B \subset \mathbb{P}^4$ is a cubic,
    \item if $(-K_B)^3 =16$ and $B$ is a double cover of $\mathbb{P}^3 $ branched along a smooth quartic, 
    \item  if $(-K_B)^3 =8$ and $B$ is the double cover of the cone over the Veronese surface (or equivalently a hypersurface of degree $6$ in $\mathbb{P}(1,1,1,2,3)$.
\end{enumerate}
\end{proposition}
\begin{corollary} Let   $B$ be a del Pezzo threefold, then 
\[\operatorname{rk} \MW(X/B)   \leq {38}.\]
If $\Sigma \subset B$ is general (in the relevant linear system),  $\operatorname{rk} \MW(X/B)   \leq {18}$.
\end{corollary}
\begin{proof} If $ (-K_B)^3 \geq 32$ by Proposition \ref{prop:delpezzogeq32} there exists a line $\ell$ with $-K_B \cdot \ell =2$, $ \ell \not \subset \Sigma$. Hence $Z$  is (birationally) a K3  
 and $\operatorname{rk} \MW(X/B)   \leq {18}$, by Theorem \ref{thm:BoundsFromNoetherPic1}.

Assume now that   $(-K_B)^3 < 32$, and let us analyze the three cases of Proposition \ref{prop:delpezzoleq32}.

In case (i),
$B \subset \mathbb{P}^4$ is  any smooth cubic,  the Fano variety of lines is two dimensional, the general line has normal bundle $\mathcal{O}_{\mathbb{P}^1} \oplus\mathcal{O}_{\mathbb{P}^1}$;  by direct computation  for every point $p \in B$ there are six lines. Then there exists a line  $ \ell \not \subset \Sigma$ with $-K_B \cdot \ell =2$.  Hence %$ \pi^{-1}(\ell)=Z \to C$ is 
$Z$  is (birationally) a K3  
 and $\operatorname{rk} \MW(X/B)   \leq {18}$ , by Theorem \ref{thm:BoundsFromNoetherPic1}.

  In case (ii), 
the double cover of a conic with 4 (distinct) bitangent points is a conic $C$ (at the ramification points the local equation is $y^2=x^2$).  The family of such conics dominates $B$, it is base point free,  the general conic is smooth  (the general such conic $C$ has either normal bundle $\mathcal{O}_X (1)\oplus \mathcal{O}_X(1)$  or $\mathcal{O}_X \oplus \mathcal{O}_X(2)$ \cite[Lemma 2.2]{MarkushevichTikhomirov}).  We then get the bound: $ \operatorname{rk} \MW(X/B)  \leq 40 - 2 =38. $
 The lines in $B$ are the inverse image of bitangent lines to the quartic surface in $\mathbb{P}^3$.  The lines do not cover $B$. If $C\subset B$ is a general line, for \textit{general} $(B, \Sigma)$,
$\quad  \operatorname{rk} \MW(X/B)  \leq18$ by 
Proposition \ref{prop:K3appearance1} (and $N_{C/X}= \mathcal{O}_{\mathbb{P}^1} \oplus \mathcal{O}_{\mathbb{P}^1} $ see also  \cite[Lemma 2.1]{MarkushevichTikhomirov}).

In case (iii),  
the conics are dominant very free movable curves \cite[Proposition 3.2] {ShimizuTanimoto}. 
Then,
$\operatorname{rk} \MW(X/B)  \leq  40 -2 =38$.  \end{proof}

\vskip 0.4in

\subsection{$B$ smooth Fano threefold  with $\operatorname{rk} \Pic (B)=1$, $r(B)=1$.}~

\medskip

\subsubsection{$-K_B$ is very ample}~
\vskip 0.1in

We start with the following: 
\begin{proposition} Let $-K_B$ be very ample with  $ 4 \leq (-K_B)^3 \leq 18$. Assume in  addition that if $ 4 \leq (-K_B)^3 \leq 8$  then $B$ is {general} in moduli. Then 
\[\operatorname{rk} \MW(X/B)  \leq 38.\]
    
\end{proposition}

\begin{proof}
The results in \cite[Theorem 1 and Lemma 8.1]{Lehmann2018RationalCO, Lehmann2025CorrigendaT}  imply that there exists a dominant family of 
smooth very free rational curves $C$  of degree $4$ on $B$; it follows that there exists a smooth surface $Z$ with $Z \to C$ and $\operatorname{rk} \MW(X/B)  \leq \operatorname{rk} \MW(Z/C) \leq 40 -2$.
\end{proof}

Next we consider the cases of  $(-K_B)^3=22$ and  the three cases when $-K_B$ is not  very ample.

\begin{proposition} If $(-K_B)^3=22$, then
\[\operatorname{rk} \MW(X/B)  \leq  10(-K_B \cdot C ) -2 \leq 18.\]
\end{proposition}
\begin{proof}There exists a $1$-dimensional family of lines 
and a $2$-dimensional family of  smooth conics  \cite[Lemma 3.2]{ArrondoFaenzi2006}. Through any point $p \in B$ there exists at least a conic  $C$. Since $-K_B \cdot C=2$, Proposition \ref{prop:K3appearance1} and Theorem \ref{thm:BoundsFromNoetherPic1} apply.
\end{proof}

\vskip 0.3in
\subsubsection{$-K_B$ is  not very ample}~
\vskip 0.1in

 Then $B$ is one of the   cases treated in the following three propositions:

\begin{proposition} Let $B$ be a hypersurface of degree $6$ in $\mathbb{P}(1^4,3)$ and $(-K_B)^3=2$, equivalently $B$ is the double cover of $\mathbb{P}^3$ branched over a surface of degree $6$.  Then 
   \[\operatorname{rk} \MW(X/B)  \leq 20+2 = 22 \,.\] 
\end{proposition}
\begin{proof} The double cover of the family of lines is a movable family of curves of genus $2$, with general smooth member $C$. Then $Z$ can be chosen to be smooth  and  $\operatorname{rk} \MW(X/B)  \leq  10(-K_B \cdot C ) +2g(C) -2 =20+2 \leq 22.$ \end{proof}

\begin{proposition}
 Let $B$ be the double cover of a \textit{ general} quadric in $\mathbb{P}^4$ branched along an octic. Then  
\[\operatorname{rk} \MW(X/B)  \leq 20 .\]
\end{proposition}
\begin{proof} The double covers of the family of lines, is a movable base point free family of curves of genus $1$, with general smooth member $C$.    Then $Z$ can be chosen to be smooth  and 
$$\operatorname{rk} \MW(X/B)  \leq  10(-K_B \cdot C ) +2-2 \leq 20 \,.$$ \end{proof}

 \begin{proposition}
    $B$ is   a special Gushel-Mukai threefold (the general double cover of a section of  a Plücker embedding of $Gr(2,5)$ by a codimension $3$ subspace) and $(-K_B)^3=10$. Then
    \[
	\operatorname{rk} \MW(X/B)  \leq  18.
	\]
\end{proposition}
   
\begin{proof}  By construction ($B$ is the double cover of a $\mathbb{P}^2$-bundle) there exists a  family of conics which dominates $B$; the family has dimension two  \cite[Theorem 7.3, Remark 7.4]{debarre2024quadricsgushelmukaivarieties}. We conclude by  Proposition \ref{prop:K3appearance1} and Theorem \ref{thm:BoundsFromNoetherPic1}.\end{proof}

\section{Bounds for elliptic varieties with $\omega_X \simeq \mathcal{O}_X$, 
of any dimension}\label{conjecture}
 Theorems \ref{thm:BoundsFromNoetherMov}, \ref{boundCY3} and Theorem \ref{thm-4foldgen}  with the Assumption \ref{Assumption} can be stated as
 \begin{theorem}\label{thm:boundsCY3and4Combined} Let  $ X \to B$ be as in Theorem \ref{thm:MMP}, with $\omega_X \simeq \mathcal{O}_X$, 
 $\dim X=3,4$.   Assume that Assumption \ref{Assumption} is satisfied if $\dim X=4$. Then
  \[\operatorname{rk} \MW(X/B) \leq  10 \cdot (\dim B +1) -2 \leq 10 \cdot  \dim X -2.\] 
In addition, if $B$ is birationally    a  Mori fiber space  $B' \to U$ with $\dim U >0$  and 
general      fiber  $F$, 
\[\operatorname{rk} \MW(X/B) \leq 10 \cdot  r(F)-2  \leq 10 \cdot (\dim F +1) -2 \leq 10 \cdot  \dim X -2.\]
 \end{theorem}
Here the Fano index $r(F)$ is the largest integer $a$ such that $-K_F = a H$ for $H$ an ample divisor, since $F$ is smooth, in particular Gorenstein. 
Also, for a smooth (Gorenstein) Fano variety $F$, the highest index is $r(F)=\dim F+1$ and it is attained for the projective space \cite{KobayashiOchiai1973}. \\

   We also have
 \begin{theorem}\label{thm:boundsCYanyDimCRational} Let  $ X \to B$ be as in Theorem \ref{thm:MMP}, $\Lambda \neq 0$ and $X$ Calabi-Yau. 
Assume there exists a  smooth rational curve $C \subset B^{sm}$ and one  of the following holds:
\begin{enumerate}[(a)] 
\item  there is an (irreducible) nonsingular complex variety $T$,
together with a smooth morphism
$\varphi:T\times{\mathbb P}^1_{\mathbb C}\to B$ with dense image,
contained in the smooth locus of $B$.
Regarding this as a family of morphisms $\varphi_t:{\mathbb P}^1_{\mathbb C}\to B$,
assume further that these are all embeddings,
giving a family $C_t\hookrightarrow B$ of smooth rational curves, which has no base points; let $C=C_t$ be a general curve in the family;
\item $C$ is movable;   % and that there are singular fibers
\item $\rho(B)=1$;\end{enumerate}
and  that  in cases  (b) and (c) either  $Z = \pi^{-1}(C)$ is   smooth or (birationally) a K3 surface.

Then
	\[
	\operatorname{rk} \MW(X/B) \leq 10 (- C  \cdot K_B)  -2 .
	\]
	
\end{theorem}
 
\begin{proof}  The statement follows from  Theorems \ref{main2}, \ref{thm:BoundsFromNoetherMov} and \ref{thm:BoundsFromNoetherPic1}. \end{proof}
\begin{remark}\label{remark:CharPN} If $(B, \Lambda)$ has log canonical singularities,  and  ${-K_B \cdot C > \dim B}$ for all rational curve $C \subset B$, then $B \simeq \mathbb{P}^{\dim B}$ (\cite{fujino2026CharacterizationProjectiveSpaceLengths}, which generalizes \cite{ChoMiyaokaSheperd-Barron2003}).   Then  there always exists a rational curve $C\subset B $, $B$ as in Theorem \ref{thm:birminmodel}, with  $(- C  \cdot K_B) \leq (\dim B+1)$. 
If such a curve satisfies one of the hypotheses of  Theorem \ref{thm:boundsCYanyDimCRational}, then
\begin{equation}\label{eq:boundsdim}
	\operatorname{rk} \MW(X/B) \leq 10 (- C  \cdot K_B)  -2 \le 10 \cdot (\dim B+1)  -2 = 10 \cdot  (\dim X)  -2 .  
\end{equation}
\end{remark}

However, it is not known that  curves with the properties in Theorem \ref{thm:boundsCYanyDimCRational} and  $(- C  \cdot K_B) \leq (\dim B+1)$  exist when $\dim X >3$. On the other hand, as in the proofs of Theorem \ref{thm-4foldgen}, it is also natural to consider curves of higher genus which, as we have proved,  also provide the expected bounds.

By combining Theorems  \ref{thm:boundsCY3and4Combined} and \ref{thm:boundsCYanyDimCRational} with  equation \ref{eq:boundsdim} in Remark \ref{remark:CharPN} and extrapolating from our results, it is natural to make the following 
\begin{conjecture}\label{conjecture} Let  $ X \to B$ be as in Theorem \ref{thm:MMP}, $X$ Calabi-Yau. 
Then   $B$ is birationally    a  Mori fiber space  $B' \to U$ with %$\dim U >0$  and 
      fiber  $F$ and 
\[\operatorname{rk} \MW(X/B)  \leq 10 \cdot (\dim F +1) -2 \leq 10 \cdot \dim X -2. \]

 \end{conjecture}
 Recall that  $B$ is birationally a Mori fiber space if $\dim X \leq 5$, or if  the standard conjectures of the minimal model program hold   (Theorem \ref{thm:birminmodel}).

 \appendix
 
 \section{Bounds from physics}\label{sec:Appendix}

To every elliptic Calabi-Yau $3$-fold $X$
one can associate a six-dimensional $N=1$ supergravity theory obtained as the low-energy limit of F-theory \cite{Vafa:1996xn, Morrison:1996na, Morrison:1996pp} compactified on $X$.
In this context,
 the free 
 part of the Mordell-Weil group of $X$ encodes
 the continuous abelian gauge symmetries in the effective action, while its torsion part
  constrains the first fundamental group of the non-abelian gauge group (see, for example, the reviews \cite{Weigand:2018rez,Cvetic:2018bni} for details and references). 
 Bounds on the rank of the Mordell-Weil group, or in physics terms on the rank $r_{\rm U(1)}$ of the abelian gauge group, are motivated by general arguments 
suggesting that the number of light degrees of freedom and their interactions  in any consistent
quantum gravitational theory are severely restricted. 
 While gauge and gravitational anomalies alone do admit infinite families of theories with arbitrarily high abelian gauge rank \cite{Park:2011wv}, taking into consideration constraints from quantum consistency of probe strings \cite{Kim:2019vuc} sets tight bounds on the ranks \cite{Lee:2019skh,Martucci:2022krl,Lee:2022swr,Kim:2024eoa}.

 In \cite{Lee:2019skh}  bounds on $r_{\rm U(1)}$ have been obtained for general $N=1$ supergravity theories in six dimensions by analysing probe consistency of certain solitonic strings coupling to gravity. 
 These bounds are given by $r_{\rm U(1)} \leq  22$ for theories with $T\geq 1$ tensor multiplets (subsequently improved  to $r_{\rm U(1)} \leq  20$ in \cite{Kim:2024eoa}) and $r_{\rm U(1)} \leq  32$ for $T =0$. 
  In explicit realisations of the supergravity via F-theory, the solitonic strings arise from
  D3-branes wrapping suitable curves on $B$. If the curve $C$ is rational, movable  and without base points on $B$, the associated string couples to all abelian gauge groups because $C$ has a non-zero intersection number with the height pairing divisor associated with the Shioda map. This follows by extending arguments in \cite{CoxZucker1979} to smooth Calabi-Yau threefolds \cite{Lee:2018ihr} and led, for smooth base $B$, to the conservative bounds 
  ${\rm rk}(\MW(X/B)) \leq 32$ for $B = \mathbb P^2$ and ${\rm rk}(\MW(X/B)) \leq 20$ otherwise in \cite{Lee:2019skh}. In \cite{Martucci:2022krl} these bounds were sharpened 
   to 
   \begin{equation} \label{bound-MW}
   {\rm rk}(\MW(X/B)) \leq 10\,  C \cdot (-K_B) -2 \,,
   \end{equation}
 which also holds for elliptic Calabi-Yau fourfolds with a smooth base $B$ and $C$ rational, smooth and in a movable class without base points. 
  For elliptic threefolds, (\ref{bound-MW}) implies  
the stronger bound ${\rm rk}(\MW(X/B)) \leq 28$ if $B = \mathbb P^2$ (leading to $T=0$) and  
${\rm rk}(\MW(X/B)) \leq 18$ otherwise. 
 The physics derivation of (\ref{bound-MW}) relies on a specific interpretation of the zero modes on the probe strings. For $X$ an elliptic threefold  with rationally fibered base $B$, \cite{Lee:2022swr} claims the bound ${\rm rk}(\MW(X/B)) \leq 18$ 
 and, by combining with \cite{HajoujiOehlmann}, a list of the possible torsion groups.

 \bigskip
 
 \hrule

\bigskip

% \bibliography{Writeup-Proof3.bib}{}

 \bibliographystyle{JHEP}
 \bibliography{Writeup-Proof3_final}

%  \bibliographystyle{plain}
%\bibliographystyle{plainarxiv}
%\bibliographystyle{plainurl}
% \newpage
 
\end{document}